\documentclass{article}

\usepackage{amsmath,amssymb,amsthm,graphicx,color,sidecap,wrapfig}
\usepackage[inner=32mm,outer=32mm,top=28mm,bottom=28mm]{geometry}
\usepackage{framed,color,perpage,textcomp}
\usepackage[toc]{appendix}
\usepackage{chngcntr}
\counterwithout{figure}{section}
\counterwithout{figure}{subsection}

\numberwithin{equation}{section}

\definecolor{gray}{rgb}{0.5,0.5,0.5}

\newcommand{\ve}{\varepsilon}

\def \div {{\rm div}}

\MakePerPage{footnote}

\newtheorem{thm}{Theorem}[section]
\newtheorem{rem}[thm]{Remark}
\newtheorem{prop}[thm]{Proposition}
\newtheorem{lem}[thm]{Lemma}
\newtheorem{cor}[thm]{Corollary}

\usepackage{authblk}

\title{Bifurcation of traveling waves in a Keller-Segel type free boundary model of cell motility}

\author[1]{Leonid Berlyand}
\author[2]{Jan Fuhrmann}
\author[3]{Volodymyr Rybalko}
\affil[1]{The Pennsylvania State University, USA, e-mail: lvb2@psu.edu}
\affil[2]{Johannes Gutenberg University Mainz
	Institute of Mathematics, Germany, e-mail: fuhrmann@uni-mainz.de}
\affil[3]{B.Verkin Institute for Low Temperature Physics and Engineering, Ukraine, e-mail: vrybalko@ilt.kharkov.ua}

\begin{document}

\maketitle

%\authors{L.Berlyand, J.Furman, V.Rybalko}

\begin{abstract}
 We study a two-dimensional free boundary problem that models motility of eukaryotic cells on substrates. 
 		This problem consists of an % vectorial
 		 elliptic equation describing the flow of cytoskeleton gel 
 		 coupled with  a convection-diffusion %scalar 
 		 PDE  for the  density of myosin motors. 
 		The two key properties of this problem are (i) presence of the cross diffusion as in the classical  Keller-Segel  problem in chemotaxis and  
 		(ii) nonlinear nonlocal free boundary condition that involves curvature of the boundary.  
 		We establish the bifurcation of the traveling waves  from  a family of  radially symmetric steady states.  The traveling waves describe persistent motion without external cues or stimuli which is a signature of cell motility.
 		 We also prove existence of non-radial steady states.   
 		% We show existence of  a family of radially symmetric steady state solutions 
% 		and prove its  non degeneracy on the interval that includes the first bifurcation point to traveling wave solutions.   
% 		The main result of the work is the bifurcation of traveling waves are from radially symmetric steady state solutions at a 
% 		critical value of a parameter.
 		 Existence of both traveling waves and non-radial  steady states  is established via   Leray-Schauder degree theory applied to a Liouville-type equation
 		 (which is obtained via a reduction of the original system) in a 
 		 free boundary setting.
 		%We also show bifurcation of non-radial steady states from radial ones. 
 		%!!We study   a generalization of the well known L. eqn in a free boundary setting!! }}
\end{abstract}

\section{Introduction}
\label{sec:intro}

For decades, the persistent motion exhibited by keratocytes on flat surfaces has attracted attention from experimentalists and modelers alike.   Cells of this type are found, e.g., in the cornea and their movement is of medical relevance as they are involved in wound healing after eye surgery or injuries. Also, keratocytes are perfect for experiments and modeling since they naturally live on flat surfaces, which allows capturing the main features of their motion by spatially two dimensional models. The  typical modes of motion of keratocytes are rest (no movement at all) or steady motion with fixed shape, speed, and direction. That is why the most important solutions  will be steady state solutions (corresponding to a resting cell) and traveling wave solutions (a steadily moving cell). 

\textcolor{black}{Traveling wave solutions for cell motility models have been investigated both analytically and numerically for free boundary problems in one space dimension, e.g. \cite{RecPutTru2013,RecPutTru2015, AltDem1999},} numerically for free boundary models in two dimensions, e.g. \cite{BarLeeAllTheMog2015, VanFenEde2011}, as well as for phase field models, analytically in one dimension, e.g. \cite{BerPotRyb2016}, and numerically  in two dimensions, e.g. \cite{ShaRapLev2010,ZieSwaAra2012, ShaLevRap2012}, for an overview we refer to \cite{ZieAra2016rev, AllMog2013} and references therein. \textcolor{black}{In this work we consider a two-dimensional model that can be viewed both as an extension of the analytical one-dimensional model from \cite{RecPutTru2013,RecPutTru2015} to 2D} and as a simplified version of the computational 2D model from \cite{BarLeeAllTheMog2015}. Our objective is to study the existence of traveling wave solutions for this model. These solutions describe steady motion without external cues or stimuli which is a signature of cell motility.%,
	%, that are  a kind of simplest nontrivial solutions 
%and 
%whose existence is a  signature  of cell  motility. % -- The existence of TW solutions is a feature of the model, not of cell motility itself

%the analytical study of  this 2D model.  

\textcolor{black}{In \cite{RecPutTru2013, RecPutTru2015}, the authors introduced  a one dimensional model capturing actin (more precisely, filamentous actin or F-actin) flow and contraction due to myosin motors. They proposed a model that consists of a system of an elliptic and a parabolic equation of Keller-Segel type in the free boundary setting. It was shown in \cite{RecPutTru2013} that trivial steady states bifurcate to traveling wave solutions.
%The key mathematical observation  made in \cite{RecPutTru2013, RecPutTru2015} is that  due to the  cross-diffusion term, the PDE is, in fact, the classical Keller-Segel  type system. 
The Keller-Segel  system in fixed domains was first introduced and analyzed in \cite{KelSeg1970,KelSeg1971a,KelSeg1971b} and studied  by many authors due to its fundamental importance  in biology most notably for modeling chemotaxis. There is a vast body of literature on Keller-Segel  models  with prescribed  (fixed rather than  free) boundary, see, e.g., \textcolor{black}{\cite{PerVas2012}, \cite{CerPerSchTanVau2011},\cite{Win2016}},
%Regularization in Keller-Segel type systems...etc         with A. Vasseur Comm. Math. Sc
%Math. Sc. Vol; 10(2) (2012) 463--476.
%Branching instabilities in Hyperbolic Keller-Segel system   with F. Cerretti, C. Schmeiser, M. Tang, N. Vauchelet.  M3AS Vol. 21, Suppl. (2011) 825--842.
\textcolor{black}{\cite{DjiWin2010}}
%Djie, K., Winkler, M.: Boundedness and finite-time collapse in a chemotaxis system with volume-filling effect. Nonlinear Analysis: Theory, Methods and ...
 review \cite{HilPai2009} 
and references therein. The key issue   in such problems  is the blow up of the solutions depending on the initial data. 
%In \cite{RecPutTru2013, RecPutTru2015} the %one-dimensional Keller-Segel type system was %introduced and studied  in the free boundary setting. 
}        
% It shares the common feature of cross diffusion term as the classical Keller-Segel system for chemotaxis. For this system, supplemented by equations for the motion of the freely moving end points of the domain (an interval in 1D), they show the existence of a family of spatially homogeneous steady states and discuss the bifurcation of traveling wave solutions from this steady state family as, e.g. the total contractile force of myosin or the traction between cell and substrate (``adhesion'') is varied.

In \cite{BarLeeAllTheMog2015} a  two-dimensional free boundary model consisting of PDEs for actin flow, myosin density and, additionally, a reaction-diffusion equation for the %adhesion parameterdensity of 
cell-substrate adhesion strength was introduced  based on mechanical principles.
% An important ingredient of this model is the novel kinematic boundary condition, see \eqref{kinemcond} below. 
 Simulations of this model reveal steady state and traveling wave type solutions in two-dimensions that are compared to experimental observations of keratocyte motion on the flat surfaces. 
%at the same values of parameters, the long time behavior depending on the initial conditions. 
The steady state solutions are characterized by a high adhesion strength (high traction) whereas the moving cell solutions correspond to a low overall adhesion strength. In both cases, the adhesion strength is spatially almost homogeneous. \textcolor{black}{Therefore in this work we consider a simplified two-dimensional problem with constant adhesion strength parameter similar to the one dimensional model of \cite{RecPutTru2013, RecPutTru2015}.} We further simplify the model in \cite{BarLeeAllTheMog2015}, see also review \cite{RecTru2016} by considering 
a reduced rheology of the cytoskeleton based on the
high contrast in numerical values for shear and bulk viscosities cited in \cite{BarLeeAllTheMog2015}.
%the parameter regime where the ratio between shear and bulk viscosities is vanishingly small,
% leading to the Darcy law for actin flow.
%\texttt{would not argue like this since the limit is singular -- size does not matter; also: rather use ``force balance'' in place of Darcy -- people might ask what our porous medium and what our fluid is} 
%This simplification is
%based on the comparison of numerical values cited in \cite{BarLeeAllTheMog2015}.
 Thus following \cite{Mog_private} we consider equations $[S1]-[S2]$ from \cite{BarLeeAllTheMog2015} with shear viscosity $\mu=0$ and bulk viscosity $\mu_b$ scaled to $1$. 
 % Note that if $\mu=0$ yields
%\textcolor{blue}{private communication with Mogilner}. 
%$\zeta$  

%	
%the evolution of the shape of the domain in two dimensional model is more complicated. In %\cite{BarLeeAllTheMog2015} this evolution is modeled by introducing special kinematic %boundary conditions which present a challenge for analysis.   

The main building block of the model considered in this work  
 %heart of the model, in dimension-free variables, 
 is a coupled Keller-Segel type system of two partial differential equations.  The first one 
 (obtained after the above simplification of equation $[S1]$ from \cite{BarLeeAllTheMog2015})
  in dimension-free variables writes as follows:
 \begin{equation}
 \label{actinflow}
 \nabla\, {\rm div}\, u +\alpha \nabla m =u\quad   \mbox{in }\Omega(t), ~~~~~~
 %-\Delta S + S = \alpha m\quad   \mbox{in }\Omega_t,
 \end{equation}
  where $\Omega(t)$ is the %moving %-- not only moving but in general also changing in shape
time dependent domain in $\mathbb{R}^2$ occupied by the cell, $u$ is the velocity of the actin gel, and  $m$ is the myosin density. This equation represents the force balance between the stress in the actin gel on the left hand side and the friction (proportional to the velocity)  between the cell and the substrate on the right hand side. Since the shear viscosity $\mu=0$, the  stress $S$  is a scalar  %split into
  composed of a hydrodynamic (passive) part $\mathrm{div}\, u$ and the active contribution $\alpha m$ generated by myosin motors. Identifying $S$ with  the corresponding  scalar matrix (tensor $S \mathbb{I}$), equation \eqref{actinflow} can be rewritten in the  standard form $\div S=u$.
%  represents the force balance between the stress caused by the pressure $\div u$ of %(F-)actin flow, stress $\alpha m$ due to myosin, $m$ being the myosin density, and the %adhesion force given by the right hand side of the equation.
% 
%  $\zeta u$ ($\zeta$ is rescaled to $1$):  
%
Equation \eqref{actinflow} is coupled to an advection-diffusion equation for the myosin density $m$: 
\begin{equation}
\label{myosindensity}
\partial_t m = \Delta m - {\rm div}(u \,m)\quad \mbox{in }\Omega(t).
% \partial_t m = \Delta m - {\rm div}(\nabla S \, m)\quad \mbox{in }\Omega_t.
\end{equation}
 Myosin motors are transported with the actin flow if bound to actin and freely diffuse otherwise, reflected by the second and first term on the right hand side of \eqref{myosindensity}, respectively. Assuming that the time scale for binding and unbinding is  very short compared to those relevant for our problem, the densities of bound and unbound myosin motors can be combined into the effective density $m$ (see e.g. \cite{RecPutTru2013, RecPutTru2015}).

%\textcolor{red}{While the boundary in  one dimensional models (e.g.\cite{RecPutTru2013, RecPutTru2015}) consists of just two points, 
%true features of a free boundary problem appear
%in two dimensional free boundary models  the shape of the domain is unknown. This  poses questions that  do not arise in one dimensional settings.  For example, bifurcations from radially to non-radially shapes.}     
%, which leads to a challenge for analysis. 
%The evolution of the domain $\Omega_t$ 
Following \cite{BarLeeAllTheMog2015}, the evolution of the free boundary $\partial\Omega(t)$ is described  by the kinematic boundary condition for the normal velocity $V_\nu$,
% of $\partial \Omega_t$ 
\begin{equation}
\label{kinemcond}
V_\nu= (u\cdot  \nu)-\beta \kappa +\lambda
\quad\mbox{on }\partial\Omega(t),
\end{equation}  
where $\nu$ is the  unit outward normal,  $\kappa$ stands for the curvature of $\partial \Omega(t)$, and constant  $\lambda$   defined by $\lambda:=\left(2\pi\beta-\int_{\partial \Omega(t)} (u\cdot  \nu) d\sigma\right)/|\partial \Omega(t)|$  enforces  area  preservation. 
%
%Here, equation \eqref{actinflow} represents the force balance between the stress of the actin gel on the left hand side and the friction between cell and substrate on the right hand side. The stress is split up into a hydrodynamic (passive) part $\mathrm{div}\, u$ and the active contribution $\alpha m$ generated by myosin motors.  To relate \eqref{actinflow} to classical  Darcy's law we rewrite it as $\nabla p=
%\zeta u$, where the pressure (deviatoric stress) is written as $p=\div u + \alpha m$ and  $\zeta=1$. 
The kinematic condition \eqref{kinemcond}  equates  the normal velocity $V_\nu$ of the boundary to  the 
%is  the balance between the   
contributions from  the normal component $(u\cdot\nu)$ of the actin velocity, the surface tension $\beta \kappa$ of the membrane ($\kappa$ being the curvature),  and the area preservation term $\lambda$. The latter term is constant along the boundary and is interpreted as actin polymerization at the membrane, it compensates for the difference between velocities  of the actin gel and the membrane. 
% which allows actin gel and membrane to move with different velocities without loosing contact. 

On the boundary, equation \eqref{actinflow}  is supplied with the zero stress condition
\begin{equation}
\label{zerostress}
\div u +\alpha m=0\  
% \frac{\partial m}{\partial \nu} =(u\cdot \nu)m\quad
 ~~  \mbox{on }\partial\Omega(t).
 %\frac{\partial S}{\partial \nu}\quad   \mbox{on }\partial\Omega_t.
\end{equation} 
whereas for the equation \eqref{myosindensity}, a no-flux condition is assumed:
\begin{equation}
\label{eq:noflux_myosin}
\frac{\partial m}{\partial \nu} =(u\cdot \nu)m\quad ~~~  \mbox{on }\partial\Omega(t).
\end{equation}

Similar parabolic-elliptic free boundary problems frequently occur in modeling of biological and physical phenomena. One type of problem arises in tumor growth models, e.g. \cite{FriRei2001,Fri00, HaoHauHuLiuSomZHa12,HuaZhaHu17} (see also reviews \cite{Fri12,LowFriJinChuLiMacWisChi09}), however, these are typically linear problems, and the domain area is not preserved.
%  on the boundary does not enter. 
For these models, steady state solutions have been described, and bifurcations to different steady states or growing/shrinking domain solutions have been investigated. \textcolor{black}{Another type of problem  arises in the modeling of  wound healing, see, e.g., \cite{JavVerVuiZwa2009}, where a free boundary problem for a reaction diffusion equation is used to model the evolution of complex wound
morphologies}. These models are often agent based rather than continuum models, see, e.g., \cite{ByrDra2009}. %The continuum limit usually lacks the nonlinear coupling via cross-diffusion, and the domain size is not conserved.
 More recently, mechanical tumor models have been devised leading to Hele-Shaw type problems, e.g. \cite{PerTanVau2014}. 

In the above works the focus is on solutions describing motion with constant velocity in domains that expand or contract rather than domains of fixed size and shape moving  with constant velocity. Besides this shift of focus, the main novelty of the  free boundary problem under consideration  is  the cross diffusion term in equation \eqref{myosindensity} giving rise to the Keller-Segel structure of the bulk equations.  \textcolor{black}{This structure was  introduced  in one dimensional models of cell motility in \cite{RecPutTru2013, RecPutTru2015}.}
\textcolor{black}{While the boundary in  one dimensional models (e.g. \cite{RecPutTru2013, RecPutTru2015}) consists of just two points, 
	%true features of a free boundary problem appear
	in two dimensional free boundary models  the shape of the domain is unknown. This  poses questions that  do not arise in one dimensional settings and leads to novel challenges in analysis, for example, bifurcations from radially symmetric to non-radially symmetric shapes. 
	%In particular, in this work  existence of non-radially symmetric steady states was established.
	}

We are interested in traveling wave solutions of \eqref{actinflow} - \eqref{kinemcond}, i.e. 
solutions of the form $\Omega_t=\Omega+Vt$, $u=u(x-V_xt,y-V_yt)$, $m=m(x-V_xt,y-V_yt)$. 
Thus after passing to the moving frame and rewriting system \eqref{actinflow}-\eqref{eq:noflux_myosin} in terms of the scalar stress $S:={\rm div} u +\alpha m$ we are led to the 
following free boundary problem
%Substitute the traveling wave ansatz 
 %\label{myosindensity}
\begin{equation}
\label{tw_actinflow}
-\Delta S + S = \alpha m\quad \mbox{in }\Omega,\ \mbox{and}\quad  S=0 \quad \mbox{on}\ \partial \Omega, 
\end{equation}
\begin{equation}
\label{tw_myosindensity}
-\Delta m +{\rm div}((\nabla S-V) m)=0 \quad   \mbox{in }\Omega,\ \mbox{and}\quad  
\frac{\partial m}{\partial \nu} =((\nabla  S-V)\cdot \nu)m\ 
\quad   \mbox{on }\partial\Omega,
\end{equation}
%with boundary conditions 
\begin{equation}
\label{tw_boundary} V_\nu=\frac{\partial S}{\partial \nu}-\beta \kappa +\lambda
 \quad   \mbox{on }\partial\Omega.
% \frac{\partial S}{\partial \nu}\quad   \mbox{on }\partial\Omega_t.
\end{equation} 

% \textcolor{red}
% {MAIN RESULT WITHOUT DETAILS, LIKE THIS:
% 	Given $R>0$, for all $\alpha>0$ and almost all $\beta>0$ there exists  a family of traveling wave solutions \eqref{tw_actinflow}-\eqref{tw_boundary} (with continuum of velocities) bifurcating form
% a radially symmetric steady state solution with $\Omega$ being the disk with radius $R$. 
% }	

We now outline the main  result of the paper (see, Section \ref{sec:bifurcation_via_degree} for further details) and key ingredients  of the proof.
% %In KS with fixed boundaries, subcritica case, mass is conserved and therefore the SS soln is nontrivial.
\begin{thm}\label{thm:main} There is a family  of  traveling waves solutions of \eqref{tw_actinflow}-\eqref{tw_boundary} with nonzero velocities $V$, bifurcating from radially  symmetric steady state solutions. This family exists for  all values of parameters $\alpha>0$ and $\beta>0$ (except, possibly, for a countable number of values of $\beta$, see Theorem \ref{thm:bifurc_teady_st}) and for any domain area $|\Omega|$. 
\end {thm}
%\textblue{
%	Given any cell size $\vert \Omega\vert$ and $\alpha>0$, and all but at most countably many values of $\beta$, there exists a family of traveling wave solutions with nonzero velocity $V$, bifurcating from a family of radially symmetric steady state solutions (solutions of \eqref{liouvillewithv}-\eqref{paraconditions2} with $V=0$ and $\Omega$ being a disk).}

Without loss of generality we assume motion in $x$-direction and, slightly abusing notation, write $V=(V,0)$. 
Furthermore, for a given  $S$ all nonnegative solutions of \eqref{tw_myosindensity}
($m$ represents the density of myosin and therefore cannot be negative) are given by $m(x,y)=m_0 e^{S(x,y)-xV}$  with some constant $m_0\geq 0$. 
Indeed, it is straightforward that $m=e^{S(x,y)-xV}$ is a solution of \eqref{tw_myosindensity}. The uniqueness up to a multiplicative constant follows from  the Krein-Rutman 
theorem \cite{KraLifSob1989}% \cite{KreRut1948,KreRut1950}, 
, or alternatively using the factorization $m=m_0(x,y) e^{S(x,y)-xV}$, considering $m_0$ as a unknown function, and 
proving that $m_0=\rm{const}$ by showing that it satisfies an advection-diffusion equation with zero Neumann 
condition. This allows us to eliminate $m$ from  \eqref{tw_actinflow}-\eqref{tw_myosindensity} and rewrite the problem of finding traveling waves 
in the following concise form:  
\begin{equation}
\label{liouvillewithv}
-\Delta S + S = \Lambda e^{S-xV}
\quad\mbox{in }\Omega,
\end{equation}
with boundary conditions 
\begin{equation}
\label{paraconditions1}
S=0\quad \mbox{on }\partial\Omega 
\end{equation}
and 
\begin{equation}
\label{paraconditions2}
V\nu_x=\frac{\partial S}{\partial \nu}-\beta \kappa +\lambda \quad   
\mbox{on }\partial\Omega.
\end{equation}
\textcolor{black}{Note that an ODE similar to the PDE \eqref{liouvillewithv} was obtained in the analysis of the one dimensional free boundary problem  for the Keller-Segel type system in  \cite{RecPutTru2013, RecPutTru2015}.}
In problem \eqref{liouvillewithv}-\eqref{paraconditions2} $S$, $V$, and $\Lambda=m_0\alpha\geq 0$ are unknowns and the parameter $\beta$ is given. 
%In equation  \eqref{liouvillewithv}  $\Lambda=m_0\alpha\geq 0$ is an additional  unknown parameter. 
Note that \eqref{liouvillewithv}-\eqref{paraconditions2} is a free boundary problem, that is, the  domain $\Omega$ is also unknown.
%, which is why it contains two boundary conditions. 
%for \eqref{actinflow}. 
For radially symmetric solutions of \eqref{liouvillewithv}-\eqref{paraconditions1}
with $V=0$ and  $\Omega$ being a disk, the constant $\lambda$ can always be chosen so that the boundary condition \eqref{paraconditions2} is satisfied.  This observation allows us to construct a one-parameter family of radially symmetric steady state solutions by solving the nonlinear eigenvalue  problem \eqref{liouvillewithv}-\eqref{paraconditions1}. Furthermore, the equation \eqref{liouvillewithv} contains exponential nonlinearity, as in the classical Liouville equation \cite{Lio1853}  which has explicit radially symmetric solutions, but the 
additional zero order term $S$ in the left hand side of \eqref{liouvillewithv} complicates the analysis. Note that non-trivial steady states also exist in the one-dimensional case  \cite{RecPutTru2013, RecPutTru2015} (they are  unstable).  

 We rely on an argument from \cite{CraRab1971} (see also \cite{Kor2012}) based on the Implicit Function Theorem  to show existence of an analytic curve $\mathcal{A}_1$ of radially symmetric solutions of \eqref{liouvillewithv}-\eqref{paraconditions1}. Moreover these solutions are extended to the case of nonzero $V$ in \eqref{liouvillewithv} and small perturbations of the domain $\Omega$ from a given disk. 
Then \eqref{liouvillewithv}-\eqref{paraconditions2} is reduced to selecting  solutions
of  \eqref{liouvillewithv}-\eqref{paraconditions1} that satisfy \eqref{paraconditions2}. 
Considering the linear part of perturbations of radially symmetric solutions we (formally) derive the condition \eqref{bifurccondition1} (Section \ref{sec:bifurcation_cond}) for a bifurcation from the steady states to genuine traveling waves (with $V\not =0$). We next show that the condition
\eqref{bifurccondition1} is indeed satisfied on a nontrivial radially symmetric  steady state solution belonging to
$\mathcal{A}_1$, exploiting a  subtle bound on the second eigenvalue of the linearized problem for the Liouville equation from \cite{Suz1992}. Yet another technically involved part of this work is devoted to recasting \eqref{liouvillewithv}-\eqref{paraconditions2} as a fixed point problem in an appropriate functional setting. Then a topological argument based on Leray-Schauder degree theory  rigorously justifies the  existence of  traveling waves with $V\not=0$.   Both the recasting and the topological argument require spectral analysis of various linearized operators appearing in these considerations. Next the techniques 
developed for establishing traveling waves solutions are also used to find  steady states with no radial symmetry.
% lacking radial symmetry.

%Thanks to the free constant $\lambda$ the condition \eqref{paraconditions2} is automatically satisfied in the particular case 
%of radially symmetric solutions of \eqref{liouvillewithv}-\eqref{paraconditions1}
%with $V=0$ when $\Omega$ is a disk.
% For a given $\Omega$ and $V$ problem 
%\eqref{liouvillewithv}-\eqref{paraconditions1} is a semilinear eigenvalue 
%problem.  

%
%has the structure of the 
%classical Liouville equation \cite{Lio1853} extansively  
%has a family of radially symmetric 
%solutions 

%\textcolor{red}{While the boundary in  one dimensional models (e.g.\cite{RecPutTru2013, RecPutTru2015}) consists of just two points, 
%	%true features of a free boundary problem appear
%	in two dimensional free boundary models  the shape of the domain is unknown. This  poses questions that  do not arise in one dimensional settings.  For example, bifurcations from radially to non-radially shapes.}     
%%, which leads to a challenge for analysis.

The rest of paper is organized as follows. In Section \ref{sec:steady_states} we find a one parameter family of radially symmetric steady state solutions and  establish  their properties.
%\textgray{These are found as solutions to the nonlinear eigenvalue problem \eqref{liouvillewithv}-\eqref{paraconditions1} with $V=0$ in a domain being prescribed as a disk of radius $R$.} \textblue{
In Section \ref{sec:bifurcation_cond} we derive a necessary condition  \eqref{bifurccondition1} for the bifurcating from the family of radially symmetric steady states to  a family of traveling wave solutions ($V\neq 0$) and show that this condition  is satisfied 
on the  analytic  curve $\mathcal{A}_1$ of radially symmetric solutions.
%
%
% family by identifying points on the steady state curve at which the linearization of \eqref{liouvillewithv}-\eqref{paraconditions2} has nontrivial solutions meaning that the implicit function theorem fails. 
%\textgray{The idea is to find solutions of \eqref{liouvillewithv}-\eqref{paraconditions1} in a domain being close to but not quite a disk and formulating a condition for these solutions to satisfy \eqref{paraconditions2} at least to linear order. This will be a condition on the steady state about which we linearized and we will show that this condition is indeed satisfied by at least one steady state in the family found in section \ref{sec:steady_states}.} 
In Section \ref{sec:fourier} we investigate the spectral properties of the linearized operator of the equation \eqref{liouvillewithv} around radially symmetric steady states. This operator appears in a number of the subsequent constructions.
%in order to distinguish nontrivial eigenvectors with $V=0$ (possibly giving rise to other %steady states bifurcating from the radially symmetric ones) from those with $V\neq 0$ which indicate possible bifurcations to genuinely traveling waves. \textgray{Here, we will identify those discrete values of the parameter $\beta$ at which a bifurcation to other steady states might interfere with the one to traveling waves.} 
In section \ref{sec:solution_dirichlet_with_v}  we establish existence of the solutions to the Dirichlet problem \eqref{liouvillewithv}-\eqref{paraconditions1} and study their properties. This is done for small but not necessarily zero velocity $V$ in  a prescribed domain $\Omega$, which is a  perturbation of a  disk. 
%\textgray{These are obtained as perturbations of the radially symmetric steady state solutions.} 
Section \ref{sec:bifurcation_via_degree} completes the proof of the main result on the bifurcation of the steady states to traveling waves. 
To this end we  rewrite  \eqref{liouvillewithv}-\eqref{paraconditions2} as a fixed point problem,  and  study the local Leray-Schauder index of the corresponding mapping.  
We show that this index  jumps  at the potential bifurcation point (identified in Section \ref{sec:bifurcation_cond}). This  establishes the  bifurcation  at this point.
%of a mapping
%Next, we shall show that as we move along the steady state curve, the local Leray-Schauder degree of the steady state solution jumps at the potential bifurcation point identified in section \ref{sec:bifurcation_cond}. This is done in section \ref{sec:bifurcation_via_degree} and by the degree jump principle we will finally have shown that indeed a bifurcation takes place at this point.} 
%\textgray{By the results from section \ref{sec:fourier} we eventually conclude that the %bifurcating branch is almost surely (except possibly for the discrete values of $\beta$ %identified there) one of genuine traveling waves.} 
Finally, in section \ref{sec:nonradial_steady_states}  we  prove existence of  nonradial steady states.
In the Appendix A we construct three terms of the asymptotic  expansion of  traveling wave solutions in powers of small velocity, which allow us  to   describe  the emergence of non-symmetric shapes both analytically and numerically. 

% % % % %TO DO
% references: KS original, find add,  perthame two  just add already found. 
% Krein Routnam in Russian and English.

\section{Family of radially symmetric steady states}
\label{sec:steady_states}

%\textblue{The steady state problem, being the traveling wave problem \eqref{liouvillewithv}-\eqref{paraconditions} with $V=0$ admits a family of 
%solutions found by imposing radial symmetry. To this end} let $\Omega$ be a disk $B_R$ with radius 
%$R>0$ and look for radially symmetric solutions $S=\Phi(r)$, $r=\sqrt{x^2+y^2}$ of \eqref{liouvillewithv} which becomes a two point boundary value problem 
Problem \eqref{liouvillewithv}-\eqref{paraconditions2} has a family of steady 
solutions, with $V=0$, found in a radially symmetric form. Namely, let $\Omega$ be a disk $B_R$ with radius 
$R>0$, then we seek radially symmetric solutions $S=\Phi(r)$, $r=\sqrt{x^2+y^2}$, of the 
equation
\begin{equation}
\label{radialLiouville}
-\frac{1}{r}( r \Phi^\prime(r))^\prime +\Phi=\Lambda e^{\Phi},\quad 0<r<R,
\end{equation}  
with boundary conditions
\begin{equation}
\label{radialLiouvilleBC}
\Phi^\prime(0)=\Phi(R)=0.
\end{equation}  
Note that \eqref{radialLiouville}-\eqref{radialLiouvilleBC} is a nonlinear eigenvalue problem, i.e. both the constant $\Lambda$ and the 
function $\Phi(r)$ are unknowns in this problem.
Every solution of \eqref{radialLiouville}-\eqref{radialLiouvilleBC} also satisfies 
\eqref{liouvillewithv}-\eqref{paraconditions2} with $V=0$ and some constant $\lambda$, that is we can always choose $\lambda$ in this radially symmetric problem,  so that the condition  \eqref{paraconditions2} is satisfied.
%, thanks to the fact that in this radially symmetric case the second condition in %\eqref{paraconditions} is satisfied by adjusting the constant $\lambda$. 
Equation \eqref{radialLiouville} is  the classical Liouville %(\textcolor{red}{Emden-Fowler})
equation \cite{Lio1853}  with an additional  zero order term (the second term on the left hand side of \eqref{radialLiouville}). Various forms of the Liouville equation arise in many applications ranging from the geometric problem of prescribed Gaussian curvature to the relativistic Chern-Simons-Higgs model \cite{NolTar1998}, the mean field limit of point vortices of Euler flow \cite{CalLioMarPul1992} and the Keller-Segel model of chemotaxis \cite{Wol1997}.  
For a review of the literature on Liouville type equations we address the reader  to
\cite{Lin2007} and references therein.  While the  above works mostly address  the  issues related to the blow-up in the  Liouville equation, see e.g., \cite{Li1999},  in contrast  our focus is on the construction of  the  family of solutions and its properties. Since we are concerned with  special solutions of \eqref{actinflow}-\eqref{eq:noflux_myosin}  such as  traveling waves and steady states  rather  than  general properties of  this evolution problem,  the issue of blow-up does not arise.

%However this zero order term complicates the analysis of \eqref{radialLiouville}-\eqref{radialLiouvilleBC} even in the radially symmetric setting. 
%\textcolor{red}{For instance, \eqref{radialLiouville} is not scale invariant unlike the Liouville equation. This lack of scale invariance does not allow us to use the standard technique based on the Pohozhaev identity, see,e.g., \cite{Kor2012},  for proving the crucial non-degeneracy condition (condition \eqref{nondegcondit} below) in the case of singular solutions.}

%Also semilinear equations of type \eqref{radialLiouville}-\eqref{radialLiouvilleBC}  typically possess several solutions, therefore below we use the concept of {\it minimal solution}.  

The following theorem establishes existence of solutions of problem \eqref{radialLiouville}-\eqref{radialLiouvilleBC}, and the subsequent lemma lists some of their properties.
%\textcolor{red}{radial symmetry is never used in thm. }
%
%, most importantly existence of a solution family that analytically depends on a parameter and whose linearized operator 
%has positive second eigenvalue.

%and satisfies so-called non-degeneracy condition (se). 
% \textcolor{red}{Next theorem in not new.}
\begin{rem} It is  natural to expect that the set of solutions of \eqref{radialLiouville}-\eqref{radialLiouvilleBC} has the same structure as the explicit  solutions of the classical Liouville equation \cite{Suz1992} in the disk. However,  the presence of the additional term $S$ in \eqref{radialLiouville} complicates 
the analysis even in the radially symmetric case, in particular, the standard  trick based  of Pohozhaev  identity no longer can be used to establish non-degeneracy (see condition \eqref{nondegcondit}). %Therefore the proof of nonexistence of the second turning point of solutions family is unknown.	 
\end{rem}

\begin{thm} \label{firstExistTH}
	Fix  $R>0$, then 
	
	 \emph{(i)} There exists a continuum (a closed connected set) $\mathcal K\subset \mathbb{R}\times C([0,R])$ of nonnegative solutions $\Lambda\geq 0$, $\Phi\geq 0$ of \eqref{radialLiouville}-\eqref{radialLiouvilleBC}, emanating from the trivial solution $(\Lambda,\Phi)=(0,0)$. 
	%	of \eqref{radialLiouville}-\eqref{radialLiouvilleBC}, 
	There exists a finite positive
\begin{equation*}
\Lambda_0=\max \{\Lambda\,|\, %(\Lambda,\Phi)\in\mathcal{K}
\text{\eqref{radialLiouville}-\eqref{radialLiouvilleBC} has a solution } (\Lambda,\Phi)
\},
\end{equation*}
in particular,  $\Lambda\leq \Lambda_0$ for all  $(\Lambda,\Phi)\in\mathcal{K}$. On the other hand $\|\Phi\|_{C([0,R])}$ is not bounded in $\mathcal K$, and moreover 
 \begin{equation}
 \label{exponentblowup1}
 \sup\left\{\int_0^R e^\Phi \, rdr \,| \, (\Lambda,\Phi)\in \mathcal{K}  \right\}=\infty.
 \end{equation}

\emph{(ii)} For every $0\leq \Lambda<\Lambda_0$ there exists a pointwise  minimal solution 
$\Phi$ (solution which takes minimal values at every point among all solutions) of 
\eqref{radialLiouville}-\eqref{radialLiouvilleBC}, 
and these 
minimal solutions are pointwise increasing in $\Lambda$. They form an analytic curve $\mathcal{A}_0$ in $\mathbb{R}\times C([0;R])$ which can
be extended to an analytic curve $\mathcal{A}_1$. The curve $\mathcal{A}_1$  is the connected component of $\mathcal{A}$  that contains $\mathcal{A}_0$,
%
%
%
%They form an analytic curve in $\mathbb{R}\times C([0,R])$ which can be extended to a connected analytic curve $$\mathcal{A}_1:=\text{connected component of $\mathcal{A}$ containing minimal 
%	solutions of \eqref{radialLiouville}-\eqref{radialLiouvilleBC}},
%	$$   
where 
\begin{equation}
\label{DefofmathcalA}
\mathcal{A}:=\{(\Lambda,\Phi)\in \mathcal{K}\,|\, \,\sigma_2(\Lambda,\Phi)>0 \},
\end{equation} 
and $\sigma_2(\Lambda,\Phi)$ denotes the second eigenvalue of the linearized eigenvalue problem
\begin{equation}
\label{linearizedSW}
\begin{matrix} -\Delta w+w-\Lambda e^\Phi w=\sigma w \quad\text{in}\  B_R,\quad w=0 \quad\text{on}\  \partial B_R.
\end{matrix}	
\end{equation}

%\textcolor{red}
%{Continuum $\mathcal{K}_1$.}
% (iii) Each solution of \eqref{radialLiouville}-\eqref{radialLiouvilleBC}  with $\Lambda\geq 0$
% satisfies
% \begin{equation}
% \label{proizvodnaya}
% \Phi^\prime(r)<0\quad \text{for}\  0<r\leq R. 
% \end{equation} 
% Also, 
% the following Pohozhaev equality holds,
% \begin{equation}
% \label{Pohozhaev}
% \frac{1}{2}(R\Phi^\prime(R))^2+\int_0^R\Phi^2\, r dr=2\Lambda \int_{0}^R e^{\Phi}r{\rm d}r-\Lambda R^2,
% \end{equation}
% or
% \begin{equation}
% \label{Pohozhaev1}
% \frac{1}{2}(R\Phi^\prime(R))^2+\int_0^R\Phi^2\, rdr=-\Lambda \int_{0}^R e^{\Phi}\Phi^\prime \, r^2 dr.
% \end{equation}
%(iii)  If both  $\Lambda$ and $\|\Phi\|_{L^\infty(0,R)}$ are small,  then  solutions $\Phi\geq 0$, $\Lambda> 0$ of \eqref{radialLiouville}-\eqref{radialLiouvilleBC}  form a curve which can be extended to an analytic curve parametrized by  e.g., numbers $s$, $0\leq s<1$, and such that the second eigenvalue $\lambda_2$ of the linearized operator 
%\begin{equation*}
%\label{linearizedEq}
%-\Delta v+v-\Lambda e^\Phi v=\lambda_2 v\quad\text{in}\  B_R,
%\end{equation*}
%\begin{equation*}
%\label{linearizedBC}
%	v=0\quad\text{on}\  \partial B_R
%\end{equation*}
%is positive for $0\leq s<1$. Moreover, either $\|\Phi\|_{C([0,R])}\to \infty$  as $s\to 1$, or 
%the solutions curve can be extended to $s=1$ and at $s=1$ the  linearized operator \eqref{linearizedEq}-\eqref{linearizedBC} has zero second eigenvalue. 
\end{thm}

\begin{rem}
           Summarizing part $(ii)$ of the theorem we have the following inclusions
           \begin{align*}
           &~~~~~~~~~~~~~\mathcal{K} & \supseteq &~~~~~~~~~\mathcal{A} &\supseteq&~~~~~~~~~~~~~\mathcal{A}_1 &\supseteq&~~~~~~~~\mathcal{A}_0\\
           &\text{continuum of solutions} && 2^{nd} \text{e.v. positive} && \text{component containing }\mathcal{A}_0 && \text{minimal solutions}             
           \end{align*}
          where at most $\mathcal{A}$ may be disconnected. 
          %If $\mathcal{A}$ is connected as well, then it coincides with $\mathcal{A}_1$ and each of the curves is a segment of the next larger one.
          The theorem establishes existence of the analytic curve of radial solutions $\mathcal{A}_1$ along which bifurcations to traveling waves with nonzero velocity occur (see Lemma \ref{est_li_zhizn_na_Marse}).
          %a sufficiently large analytic curve of solutions along which we can search bifurcations to traveling wave solutions with nonzero velocity.
          \end{rem}

%\begin{rem}
%Recall that the pairs $(\Lambda, \Phi)$  where zero is an eigenvalue of \eqref{linearizedSW} are called singular \cite{Kor2012} since  they could lead to bending, terminating  or bifurcation  of the solution curve.   
%\end{rem}

\begin{proof}
	\emph{(i)} By the maximum principle every solution of \eqref{radialLiouville}-\eqref{radialLiouvilleBC} with $\Lambda\geq 0$ is positive for  $r<R$. 
	Let $\mu_D>0$ denote the first eigenvalue of $-\Delta$ in $B_R$
	with homogeneous Dirichlet boundary condition, and let $U>0$ be the corresponding eigenfunction. Then multiplying 
	\eqref{radialLiouville}
	%-\eqref{radialLiouvilleBC} 
	by $r U$ and integrating we find
	$$
	(1+\mu_D) \int_0^R U\Phi \, rdr= \Lambda \int_0^R e^{\Phi} U \, rdr\geq \Lambda \int_0^R \Phi U \, rdr,
	$$
	and therefore $\Lambda \leq 1+\mu_D$.
	
	%At this point it is convenient to switch to 2D formulation 
	To show the existence of the continuum $\mathcal{K}$, we rewrite   \eqref{radialLiouville}
	%-\eqref{radialLiouvilleBC}
	as 	
	\begin{equation}
	\label{semiresolved}
	-\Delta \Phi +\Phi=\tilde \Lambda \left(\frac{e^{2\Phi}}{\int_{B_R} e^{2\Phi}dxdy}\right)^{1/2}\quad \text{in}\ B_R,
	\end{equation}
	with $\Phi=\Phi(r)$, $r=\sqrt{x^2+y^2}$, and the new unknown parameter $\tilde \Lambda$ in place of $\Lambda$.  Then we resolve  
	\eqref{semiresolved} with Dirichlet condition $\Phi=0$ on $\partial B_R$, considering 
	the right hand side of \eqref{semiresolved} as a given function. This leads to an equivalent reformulation of  \eqref{radialLiouville}-\eqref{radialLiouvilleBC} as a fixed point problem of the form 
	\begin{equation}
	\label{FFFixedPPP}
	\Phi=\tilde \Lambda \mathcal{R}(\Phi).
	\end{equation}
	By standard elliptic estimates $\mathcal{R}$ is a compact mapping in $C([0,R])$, moreover $\mathcal{R}(C([0,R]))$ is a bounded subset of $C([0,R])$. Therefore we can apply  Leray-Schauder continuation arguments, see, e.g., \cite{Maw1999}, and find that there is a continuum of solutions $(\tilde \Lambda, \Phi)$ of \eqref{FFFixedPPP} emanating from $(0,0)$ and $\tilde \Lambda$
	takes all nonnegative values. Then in view of the boundedness of  $\Lambda=\tilde \Lambda/\left(2\pi \int_0^R e^{2\Phi}rdr \right)^{1/2}$ we conclude that $\sup\{\|\Phi\|_{C([0,R])};\, (\Lambda,\Phi)\in \mathcal{K}  \}=\infty$. This in turn implies \eqref{exponentblowup1} by Corollary 6 of \cite{BreMer1991}.

\emph{(ii)} According to \cite{KeeKel1974} there is a minimal solution $\Phi$ of \eqref{radialLiouville}-\eqref{radialLiouvilleBC} for each $\Lambda\in[0,\Lambda_0)$ with $\Phi$ depending monotonically on $\Lambda$. Consider now any, not necessarily minimal, solution $(\Lambda,\Phi)$ such that the second eigenvalue $\sigma_2(\Lambda,\Phi)$ of the linearized problem \eqref{linearizedSW}
%	\begin{equation}
%\label{linearizedSW}
%\begin{matrix} -\Delta w+w-\Lambda e^\Phi w=\sigma w &\quad\text{in}\  B_R, \\
%	w=0 & \quad\text{on}\  \partial B_R.
%\end{matrix}	
%\end{equation}
is positive. 
By using well-etablished techniques based on the Implicit Function Theorem, see, e.g. \cite{Kor2012}, we obtain that all the solutions of \eqref{radialLiouville}-\eqref{radialLiouvilleBC} in a neighborhood of $(\Lambda,\Phi)$ belong to a smooth 
curve through $(\Lambda, \Phi)$, provided that either the linearized problem  \eqref{linearizedSW}
has no zero eigenvalue or this eigenvalue is simple and the corresponding eigenfunction $w$
satisfies the non-degeneracy condition 
\begin{equation}
	\label{nondegcondit}
	\int_0^R e^{\Phi(r)} w(r)\, rdr\neq 0.
\end{equation} 
Since by assumption $\sigma_2(\Lambda,\Phi)>0$, the zero eigenvalue, if any exists, is the first 
eigenvalue of \eqref{linearizedSW} and therefore $w$ has a fixed sign and necessarily \eqref{nondegcondit} holds. Thus $\mathcal{A}_1$ is indeed a smooth curve, it contains the minimal solutions (those, for which the first eigenvalue $\sigma_1(\Lambda,\Phi)$ of linearized problem  \eqref{linearizedSW} is nonnegative) but extends beyond these.
%Moreover, thanks to the continuity of $\sigma_2(\Lambda,\Phi)>0$ and the set $\mathcal {K}$ being connected,  it follows that the set $\mathcal{K}_1$, defined by  \eqref{DefofmathcalA},
%%\begin{equation}
%%\mathcal{K}:=\{(\Lambda,\Phi)\in \mathcal{K})\,|\, \,\sigma_2(\Lambda,\Phi)>0 \},
%%\end{equation} 
%is a connected curve of solutions of \eqref{radialLiouville}-\eqref{radialLiouvilleBC}. This curve contains the minimal solutions (those, for which the first eigenvalue $\sigma_1(\Lambda,\Phi)$ of the linearization \eqref{linearizedSW} is nonnegative) but extends beyond these.
Finally,  since the nonlinearity $e^\Phi$ in \eqref{radialLiouville} is analytic the curve $\mathcal{K}_1$ is analytic as well, see the proof of Proposition \eqref{DefineSolLiouvilNonrad}.  
%	
%	
%	
%	This curve can be continued by the Implicit Function Theorem (see, e.g. \cite{GraRab??} ) as long as the linearized problem \eqref{linearizedEq}-\eqref{linearizedBC} does not have zero eigenvalue.
%	 
%	
% 	
% 	Item (iii). The existence of solutions of  \eqref{radialLiouville}-\eqref{radialLiouvilleBC}
% 	for small $\Lambda$ and $\|\Phi\|_{C([0,R])}$ follows from applying the contraction mapping theorem to \eqref{FFFixedPPP}. Since the nonlinearity  in \eqref{radialLiouville} is analytic, the set of constructed  solutions 
% 	%of  \eqref{radialLiouville}-\eqref{radialLiouvilleBC}  
% 	forms an analytic curve. This curve can be continued by the Implicit Function Theorem (see, e.g. \cite{CraRab1973} ) as long as the linearized problem \eqref{linearizedSW} does not have zero eigenvalue. The solution curve can be uniquely (and analtytically) continued even if \eqref{linearizedSW} has zero eigenvalue provided that the corresponding eigenfunction $v$ satisfies the non-degeneracy condition (see, e.g. \cite{Kor2012}),
% 	\begin{equation}
% 	\label{nondegcondit}
% 	\int_0^R e^{\Phi(r)} v(r)\, rdr\not =0,
% 	\end{equation}
% 	which is always satisfied for the first eigenfunction of the linearized operator thanks to the fact that this eigenfunction is positive. 
% 	
% 	
\end{proof}

\begin{lem}\label{lem:properties_steady_states}
 Each solution of \eqref{radialLiouville}-\eqref{radialLiouvilleBC}  with $\Lambda\geq 0$ satisfies
\begin{equation}
\label{proizvodnaya}
\Phi^\prime(r)<0\quad \text{for}\  0<r\leq R. 
\end{equation} 
and the following Pohozhaev equalities
\begin{equation}
\label{Pohozhaev}
\frac{1}{2}(R\Phi^\prime(R))^2+\int_0^R\Phi^2\, r dr=-\Lambda \int_{0}^R e^{\Phi}\Phi^\prime \, r^2 dr=
2\Lambda \int_{0}^R e^{\Phi}r{\rm d}r-\Lambda R^2.
\end{equation}
%and
%\begin{equation}
%\label{Pohozhaev1}
%\frac{1}{2}(R\Phi^\prime(R))^2+\int_0^R\Phi^2\, rdr=-\Lambda \int_{0}^R e^{\Phi}\Phi^\prime \, r^2 dr.
%\end{equation}
\end{lem}

\begin{proof} To show \eqref{proizvodnaya} we first prove that $\Phi(r)$ is decreasing.  Assume to the contrary that $\Phi$ takes a local minimum at $r_0$ and  there is  $r_1\in(r_0, R]$ such that $\Phi(r_0)=\Phi(r_1)$. Multiply 
\eqref{radialLiouville} by $\Phi^\prime(r)$ and integrate from $r_0$ to $r_1$ to get
\begin{equation}
\label{zchetointegriruem}
\int_{r_0}^{r_1} \left(\Phi^{\prime\prime}+\frac{1}{r} \Phi^\prime\right)\Phi^\prime\,dr=
\frac{1}{2}\Phi^2(r_1)-\Lambda e^{\Phi (r_1)}-\frac{1}{2}\Phi^2(r_0)+\Lambda e^{\Phi(r_0)}=0.
\end{equation} 
On the other hand, the left hand side of \eqref{zchetointegriruem}
is 
\begin{equation*}
\frac{1}{2}(\Phi^\prime(r_1))^2+\int_{r_0}^{r_1} \frac{1}{r} (\Phi^\prime)^2\,dr.
\end{equation*}
Therefore $\Phi$ is constant on $(r_0,r_1)$, this in turn implies that $\Phi$ is constant on $(0,R)$, a contradiction. Thus $\Phi'(r)\leq 0$ for $0<r<R$. Next, 
assuming that  $\Phi'(r_0)=0$ at a point $0<r_0<R$ we get $\Phi^{\prime\prime}(r_0)=0$. This also implies that $\Phi$ is constant
on $(0,R)$. Finally, $\Phi^\prime(R)<0$ by the Hopf Lemma. 

The equalities in \eqref{Pohozhaev} are obtained in the standard way, multiplying \eqref{radialLiouville} by the Pohozhaev 
multiplier $r^2 \Phi^\prime(r)$, then taking the integral from $0$ to $R$  and integrating by parts.
\end{proof}

\section{Necessary condition for bifurcation of traveling waves}
\label{sec:bifurcation_cond}

We seek traveling wave solutions with small velocity, i.e. solutions of \eqref{liouvillewithv}-\eqref{paraconditions2}
for small $V=\ve$, as perturbations of radially symmetric steady states given by a pair  $(\Lambda, \Phi(r))$ of solutions to \eqref{radialLiouville}-\eqref{radialLiouvilleBC}. To this end we plug the ansatz 
\begin{equation}
\label{ANSATZ}
S=\Phi +\ve \phi +\dots , \quad%~~~~~
\Omega=\{(x,y)=r(\cos\varphi,\sin\varphi) \mid \varphi \in [-\pi, \pi), r < R + \varepsilon \rho(\varphi) + \dots \}
%=(R+\ve\rho(\varphi)+\dots)\cos \varphi, y=(R+\ve \rho (\varphi)+\dots)\sin \varphi\}
\end{equation}
into \eqref{liouvillewithv}-\eqref{paraconditions2}. The function $\rho$ describes the deviation of $\Omega$ 
from the disc $B_R$ while $\phi$ describes  the deviation of the stress $S$ from $\Phi$. Note that in this first order approximation the constant $\Lambda$ is not perturbed  (see Appendix A, where it is shown that the first correction $\ve \Lambda_1=0$). 
%\textblue{where $\rho(\varphi)$ describes the deviation of $\Omega$ from a circular %shape and has to be chosen such that $\vert \Omega \vert = \pi R^2$ is the same as %for the disk $B_R$.} 
Equating like powers of $\ve$, the terms of order $\ve$ in \eqref{liouvillewithv} yields the linear inhomogeneous equation for $\phi$ : 
%This yields,  considering   terms of the  order $\ve$  in \eqref{liouvillewithv}, 
\begin{equation}
\label{firstorderapproxLiouv}
-\Delta \phi+\phi=\Lambda e^\Phi(\phi -x) \quad \text{in}\ B_R. 
\end{equation}
Furthermore, equating terms of the order $\ve$ in the boundary conditions \eqref{paraconditions1}, \eqref{paraconditions2} we get
\begin{equation}
\label{firstorderapproxLiouvBC}
 \phi+\Phi^\prime(R)\rho=0,
 % \quad \text{on}\ \partial B_R,
\end{equation}
and 
\begin{equation}
\label{firstorderapproxKinemBC}
\cos\varphi = \frac{\partial \phi}{\partial \nu}+\Phi^{\prime\prime}(R)\phi +\frac{\beta}{R^2}(\rho+\rho^{\prime\prime}).
%\quad \text{on}\ \partial B_R,
\end{equation}
%\textcolor{red}
%{The problem \eqref{firstorderapproxLiouv}-\eqref{firstorderapproxKinemBC} is stated in the fixed domain 
%$\Omega=B_R$ but still contains an extra boundary condition making the problem overdetermined. This apparent contradiction is resolved by a proper choice of a solution $(\Lambda, \Phi)$ of the nonlinear eigenvalue problem \eqref{radialLiouville}-\eqref{radialLiouvilleBC}. Specifically, the condition  }
To get rid of trivial solutions arising from  infinitesimal shifts of the disk $B_R$, 
%a given steady state, 
we require $\rho$ to satisfy the orthogonality condition 
%(see \cite{SYMMETRY-BREAKING BIFURCATION OF
%ANALYTIC SOLUTIONS TO FREE BOUNDARY PROBLEMS:
%AN APPLICATION TO A MODEL OF TUMOR GROWTH, reitich})
\begin{equation}
\label{ShiftOrthogonalCond}
\int_{-\pi}^{\pi} \rho(\varphi) \cos\varphi d\varphi =0.
\end{equation}
A solution of  \eqref{firstorderapproxLiouv}-\eqref{firstorderapproxLiouvBC} is sought in the form of the Fourier component $\phi=\tilde\phi(r)\cos \varphi $. Then, $\tilde \phi(r)$ has to satisfy
\begin{equation}
\label{equatttion81}
-\frac{1}{r}(r\tilde\phi^\prime)^\prime+\left(1+1/r^2\right)\tilde \phi=\Lambda e^\Phi(\tilde\phi -r),
\quad 0<r<R,\quad \tilde \phi(0)=0, 
\end{equation}
and, owing to \eqref{ShiftOrthogonalCond} and \eqref{firstorderapproxLiouvBC}, the boundary condition% for  $\tilde\phi$ is
\begin{equation}
\label{equatttion81BC}
\tilde\phi(R) =0.
\end{equation}

Now multiply \eqref{equatttion81} by $\Phi^\prime(r) r$ and integrate from $0$ to $R$. 
Taking into account that differentiating \eqref{radialLiouville}  yields
$ -\frac{1}{r}(r\Phi^{\prime\prime})^\prime+\left(1+1/r^2\right)\Phi^\prime =
\Lambda e^\Phi\Phi^\prime$, we  integrate by parts to obtain
\begin{equation}
R \Phi^\prime(R)=\Lambda \int_0^R e^{\Phi(r)}\Phi^\prime(r) r^2dr, 
\label{bifurccondition1}
\end{equation}
where we have also used \eqref{equatttion81BC} and \eqref{firstorderapproxKinemBC}. 
This is a necessary condition for existence of 
%a curve of 
traveling waves bifurcating from the steady state curve at the point $(\Lambda,\Phi)$, and it can be equivalently rewritten using \eqref{radialLiouville}-\eqref{radialLiouvilleBC}, 
\begin{equation}
\label{bifurccondition2}
\int_0^R\Phi(r)rdr=\Lambda R^2 -\Lambda\int_0^R e^{\Phi}rdr,
\end{equation}  
or, using \eqref{Pohozhaev},
\begin{equation}
\label{bifurccondition3}
R\Phi^\prime(R)+\frac{1}{2}\left(R\Phi^\prime(R)\right)^2+\int_0^R\Phi^2(r)rdr=0. 
\end{equation}  
The following Lemma \label{est_li_zhizn_na_Marse} shows that there  exists a pair 
$(\Lambda,\Phi)\in \mathcal{A}_1$ satisfying \eqref{bifurccondition1}, and subsequent Corollary 	\ref{OchenHorCorrolaryBifurk} specifies  such a pair which is used in the proof 
of Theorem \ref{thm:main}.
%  satisfies  \eqref{bifurccondition1} with some non-degeneracy condition.
%Next lemma there is a solution s.t. \eqref{bifurccondition1} is satisfied and non-degeneracy 
%condition is satisfied.

%\textblue{The following lemma establishes the existence of a potential bifurcation point on the segment $\mathcal{A}_1$ of the steady state curve $\mathcal{K}$ containing the minimal solutions and having at most one non-positive eigenvalue of the linearization.}

\begin{lem}\label{est_li_zhizn_na_Marse}
	There are solutions $(\Lambda_-,\Phi_-)$ and $(\Lambda_+,\Phi_+)$ of \eqref{radialLiouville} -\eqref{radialLiouvilleBC} which belong to the curve $\mathcal{A}_1$ (see item (ii) of Theorem \ref{firstExistTH}) and satisfy 
	\begin{equation}
	\label{bifurccondition2bis}
	\int_0^R\Phi_{-}(r)rdr<\Lambda_{-} R^2 -\Lambda_{-}\int_0^R e^{\Phi_{-}(r)}\, rdr,
	\end{equation} 
	\begin{equation}
	\label{bifurccondition2bisbis}
	\int_0^R\Phi_{+}(r)rdr>\Lambda_{+} R^2 -\Lambda_{+}\int_0^R e^{\Phi_{+}(r)} \,rdr.
	\end{equation} 
%    	satisfies\eqref{bifurccondition1}. 
%	Moreover, if the linearized operator about this solution 
%	can has zero eigenvalue, then the corresponding eigenfunction is sign preserving so that the non-degeneracy holds.   
\end{lem}

\begin{proof} Let us consider minimal solutions in $\mathcal{A}_1$ corresponding to small $\Lambda>0$, and small $\|\Phi\|_{L^\infty(B_R)}$.
	%so that $(\Lambda,\Phi)\to (0,0)$ as $s\to 0$. 
	%For $s\to 0$ the solution $(\Lambda,\Phi)$ converges to $(0,0)$, and 
We show that the left hand side of 
\eqref{bifurccondition2}
 is strictly less than its right hand side
by considering the leading term of the asymptotic expansion of solutions in the limit
$\Lambda\to 0$.
%We 
%study the leading term of the asymptotic expansion of solutions in the limit $\Lambda\to 0$ to show that the left hand side of \eqref{bifurccondition2} is strictly less than its right hand side. 
Linearizing  \eqref{radialLiouville} -\eqref{radialLiouvilleBC} 
about $(0,0)$
%for small $\Lambda>0$ and small $\|\Phi\|_{L^\infty(B_R)}$
we get
\begin{equation}
\label{LegkayaAsimptotika}
\Phi=\Lambda g+O(\Lambda^2),\quad  \text{where $g$ solves} \  -\frac{1}{r}(rg^\prime)^\prime+g=1,\ r<R,\quad  g^\prime(0)=g(R)=0.
\end{equation}  
By the maximum principle $0<g(r)<1$ for $r<R$,  and therefore on the left hand side of 
\eqref{bifurccondition2} we have
$$
\int_0^R\Phi(r)\,rdr=\Lambda \int_0^Rg\,rdr+O(\Lambda^2)\leq \Lambda (R^2/2-\delta)+O(\Lambda^2),
$$
for some $\delta>0$ independent of $\Lambda$, while on the right hand side of \eqref{bifurccondition2}, 
$$
\Lambda R^2 -\Lambda\int_0^R e^{\Phi}rdr=\Lambda R^2- \Lambda \int_0^R(1+\Lambda g)rdr+O(\Lambda^2)=\Lambda R^2/2 +O(\Lambda^2).
$$

Next we  show existence of $(\Lambda_+,\Phi_+)\in \mathcal{A}_1$ satisfying 
\eqref{bifurccondition2bisbis}.
%
%
%that somewhere on the curve $\mathcal{K}_1$ the left hand side of  \eqref{bifurccondition2} becomes bigger or equal to its right hand side,  and use the 
%fact that $\mathcal{K}_1$ is connected to conclude that there is pair $(\Lambda,\Phi)\in\mathcal{K}_1$ satisfying \eqref{bifurccondition2}. 

\emph{Case 1: $R\leq 4$.} According to items (i) and (ii) of Theorem \ref{firstExistTH}, the curve $\mathcal{A}_1$ satisfies
\begin{equation}
\label{exponentunbounded}
\sup\left\{\int_0^R e^\Phi \, rdr \,| \, (\Lambda,\Phi)\in \mathcal{A}_1  \right\}=\infty,
\end{equation}
or, if this is false, at least
\begin{equation}
\label{seceigenvaluezero}
\inf\left\{\sigma_2(\Lambda,\Phi) \,| \, (\Lambda,\Phi)\in \mathcal{A}_1  \right\}=0.
\end{equation}
%\textblue{\texttt{Add simulated curve here to suggest that in fact both are true?}}

If  \eqref{exponentunbounded} holds then right hand side \eqref{bifurccondition2} becomes negative, while the left hand side is positive, and we are done. 

Now consider the case that \eqref{seceigenvaluezero} holds. By continuity of $\sigma_2(\Lambda,\Phi)$ there is a 
pair $(\Lambda,\Phi)\in \mathcal{K}_1$ such that the second eigenvalue of 
\eqref{linearizedSW} is less than $1$. In other words, the second eigenvalue of   
\begin{equation}
\label{linearizedSW1}
%\begin{matrix}
-\Delta v-\Lambda e^\Phi v=\sigma v \quad\text{in}\  B_R, \quad v=0 \quad\text{on}\ \partial B_R 
%\end{matrix}
 \end{equation}
is negative. Then, according to Proposition 2 in \cite{Suz1992}, we have
 \begin{equation*}
 \Lambda \int_0^R e^{\Phi}rdr\geq 4.
 \end{equation*} 
%In the first case the right hand side of \eqref{bifurccondition2} becomes negative.
%It remains to consider the second case. 
Assume by contradiction, that the right hand side of \eqref{bifurccondition2} is  bigger than or equal to its left hand side, then in view of the equivalent reformulation \eqref{bifurccondition3}
of \eqref{bifurccondition2}, we find
\begin{equation}
\label{QQQQQUUUU01}
R\Phi^\prime(R)+\frac{1}{2}\left(R\Phi^\prime(R)\right)^2+\int_0^R\Phi^2(r)rdr<0,
\end{equation} 
which in turn implies that 
\begin{equation}
\label{QQQQQUUUU02}
R\Phi^\prime(R)>-2\quad\text{and}\quad\int_0^R\Phi^2(r)rdr\leq 1/2. 
\end{equation}
On the other hand, multiplying \eqref{radialLiouville} by $r$ and integrating we find
\begin{equation}
\label{QQQQQUUUU03}
\Lambda \int_0^R e^{\Phi}rdr=\int_0^R \Phi\, rdr-R\Phi^\prime(R).
\end{equation}
Combining \eqref{QQQQQUUUU03} with \eqref{QQQQQUUUU01} and the first inequality in \eqref{QQQQQUUUU02} we get
\begin{equation}
\label{QQQQQUUUU04}
\int_0^R \Phi \,rdr>2.
\end{equation}
Finally, applying the Cauchy-Schwarz inequality and using the second inequality in \eqref{QQQQQUUUU02} leads to
\begin{equation}
\label{QQQQQUUUU05}
\int_0^R \Phi\, rdr\leq \frac{R}{\sqrt{2}}\left(\int_0^R \Phi\, rdr\right)^{1/2}\leq \frac{R}{2}.
\end{equation}
Thus, \eqref{QQQQQUUUU04} and \eqref{QQQQQUUUU05}
yield the lower bound for the radius, $R>4$, so that the Lemma is proved for $R\leq 4$.

\emph{Case 2: $R\geq 4$}. Observe that the maximal value $\Lambda_0$ of $\Lambda$ admits the lower bound $\Lambda_0\geq 1/e$. Indeed, considering the initial value problem 
\begin{equation}
-q^{\prime\prime}-\frac{1}{r}q^\prime +q=e^{q-1},\quad r>0, ~~~
\quad q(0)=A, q^\prime(0)=0,
\end{equation}
we find that $q(R)$ continuously varies from $-\infty$ to $1$ as $A$ decreases from $+\infty$ to $1$. Therefore there exists some $A>1$ such that $\Phi=q$ is a solution of \eqref{radialLiouville}-\eqref{radialLiouvilleBC}. Now consider the minimal solution $\Phi$ 
of \eqref{radialLiouville}-\eqref{radialLiouvilleBC} with $\Lambda=1/e$
%\textcolor{red}
%{, note that $(1/e,\Phi)\in \mathcal{K}_1$ ,} 
 and introduce the function  $w$ solving the auxiliary problem 
\begin{equation}\label{BesselMod}
-w^{\prime\prime}-\frac{1}{r}w^\prime +w=(w+1)/e,\quad r>0,~~~
\quad  w^\prime(0)=w(R)=0.
\end{equation}
Since $w$ is a positive subsolution of \eqref{radialLiouville}-\eqref{radialLiouvilleBC}, 
we have
$$
\Phi\geq w\quad \text{ for } \ r<R.
$$ 
Therefore, in order to prove the inequality 
\begin{equation}
\label{EsliHochetsyaToMozhno01}
R\Phi^\prime(R)+\frac{1}{2}\left(R\Phi^\prime(R)\right)^2+\int_0^R\Phi^2(r)rdr\geq 0,
\end{equation} 
it suffices to show that 
\begin{equation}
\label{EsliHochetsyaToMozhno02}
\int_0^R w^2(r)\, rdr\geq 1/2.
\end{equation} 
The solution $w$ of \eqref{BesselMod} is explicitly given by
$$
w(r)=\frac{1}{e-1}\left\{1-\frac{I_0 (\theta r)}{I_0(\theta R)}\right\},  
$$
where $\theta=\sqrt{1-1/e}$, and $I_0$ is the modified Bessel function of the first kind. Since 
%$$
%v=\frac{1}{e-1}\Bigl(1-\frac{R \sinh (\theta r)}{r\sinh (\theta R)}\Bigr),  \ \text{where}\ \theta=\sqrt{1-1/e}.
%$$
%$$
% \frac{1}{e-1}\Bigl(1-\frac{\cosh (\theta r)}{\cosh (\theta R)}\Bigr) \geq  v\geq \frac{1}{e-1}\Bigl(1-\frac{R \sinh (\theta r)}{r\sinh (\theta R)}\Bigr),  \ \text{where}\ \theta=\sqrt{1-1/e}.
%$$
$$
J(R):=\int_0^R w^2\, r dr =\frac{1}{(e-1)^2}\left\{\frac{R^2}{2}-2R\frac{I_1(\theta R)}{\theta I_0 (\theta R)}
+\frac{R^2}{2 I_0(\theta R)^2}\left( I_0 (\theta R)^2-I_1 (\theta R)^2\right),
\right\}
$$
%where $I_1(r)=I_0^\prime (r)$, 
is increasing in $R$ and
$$
J(4)=0.78... >1/2,
$$	
the inequality \eqref{EsliHochetsyaToMozhno02}	holds for $R\geq 4$, and 
so does \eqref{EsliHochetsyaToMozhno01}. This completes  the proof of Lemma \ref{est_li_zhizn_na_Marse}.
\end{proof}

\begin{cor} 
	\label{OchenHorCorrolaryBifurk}
	There exists a pair $(\Lambda_0,\Phi_0)\in \mathcal{A}_1$  satisfying  the necessary bifurcation condition 
	\eqref{bifurccondition1}. Moreover, in an arbitrary neighborhood of  this pair
$(\Lambda_0,\Phi_0)$ %, say in the topology of $\mathbb{R}\times L^{\infty}(B_R)$, 
we find $(\Lambda_\pm,\Phi_\pm)\in \mathcal{A}_1$ such that 
\begin{equation}
 R\Phi_{-}^\prime(R)<\Lambda_{-} \int_0^R e^{\Phi_{-}(r)}\Phi_{-}^\prime(r) r^2dr, 
 \quad 
 R\Phi_{+}^\prime(R)>\Lambda_{+} \int_0^R e^{\Phi_{+}(r)}\Phi_{+}^\prime(r) r^2dr
 \label{bifurccondition_dif_sign}
 \end{equation}
 The condition  \eqref{bifurccondition_dif_sign}  shows that  $(\Lambda_0,\Phi_0)$ is  a robust root of the equation \eqref{bifurccondition1}.
\end{cor}

\begin{proof} 
The result follows from Lemma	\ref{est_li_zhizn_na_Marse} thanks 
 to analyticity of the curve  $\mathcal {A}_1$ and to the fact that $\mathcal {A}_1$ is connected.
\end{proof}

\section{Fourier analysis of the linearized operator}
\label{sec:fourier}

To construct solutions of problem 
\eqref{liouvillewithv}-\eqref{paraconditions1}
as perturbations of radially symmetric steady states we need to study the
properties of the linearized operator of this problem. Namely, we  
consider the linearized spectral problem 
\begin{equation}
\label{linearizedSpect}
%\begin{matrix}
-\Delta w+w-\Lambda e^{\Phi}w =\sigma w\quad\text{in}\  B_R, 
\quad w=0 \quad\text{on}\ \partial B_R ,
%\end{matrix}
\end{equation}
where $(\Lambda,\Phi)$ is a pair satisfying 
\eqref{radialLiouville}-\eqref{radialLiouvilleBC}.  

\begin{prop} 
	\label{Spectrum_Liouville_linaerized}
	For any $n, l=1,2,\dots$, the $l^{th}$ eigenvalue $\sigma_{nl}$ corresponding to the $n^{th}$ Fourier modes 
$w_{nl}(r)\cos n\varphi$ and $w_{nl}(r)\sin n\varphi$,% ($n$ denoting the Fourier mode, $l$ being the number of the eigenvalue), 
\begin{equation}
\label{linearizedSpecRad}
%\begin{matrix}
- \frac1r(r w^{\prime}_{nl})'+\frac{n^2}{r^2}w_{nl}+w_{nl}-\Lambda e^{\Phi}w_{nl} =\sigma_{nl} w_{nl}, \quad 0<r<R, 
\quad ~~~ w_{nl}(0)=w_{nl}(R)=0,
%\end{matrix}
\end{equation}
is positive, $\sigma_{nl}>0$.
\end{prop}	

\begin{proof}
For each $\delta>0$ and any solution $\Phi$ of \eqref{radialLiouville}-\eqref{radialLiouvilleBC}, the function $\Theta_\delta:r\mapsto \delta - \Phi^\prime(r)$ is strictly positive and satisfies 
(by differentiating \eqref{radialLiouville})
% \begin{equation}
% -\frac1r\left(r(\delta-\Phi_0')'\right)' + \left(1+\frac{1}{r^2}\right)(\delta-\Phi_0') - \Lambda e^{\Phi_0}(\delta-\Phi_0') = \left(1+\frac{1}{r^2}\right)\delta - \delta \Lambda e^{\Phi_0} ,~~~~0<r<R
% \end{equation}
\begin{equation}\label{eq:Theta_plain}
-\frac1r\left(r\Theta_\delta'\right)' +\left(1+\frac{1}{r^2}-\Lambda e^{\Phi}\right)\Theta_\delta  = \left(1+\frac{1}{r^2}-\Lambda e^{\Phi}\right)\delta,\quad 0<r<R
\end{equation}
or, for any given $n$,
% \begin{equation}
%  -\frac1r\left(r(\delta-\Phi_0')'\right)' + \left(1+\frac{n^2}{r^2}\right)(\delta - \Phi_0') = \left(1+\frac{n^2}{r^2}\right)\delta - \delta\Lambda e^{\Phi_0} - \frac{n^2-1}{r^2}\Phi_0'.
% \end{equation}
\begin{equation} \label{eq:Theta_n}
 -\frac1r\left(r \Theta_\delta'\right)' + \left(1+\frac{n^2}{r^2}-\Lambda e^{\Phi}\right)\Theta_\delta = \left(1+\frac{n^2}{r^2}-\Lambda e^{\Phi}\right)\delta - \frac{n^2-1}{r^2}\Phi^\prime,\quad 0<r<R.
\end{equation}
Multiplying \eqref{linearizedSpecRad} by $r w_{nl}$ and integrating from $0$ to $R$ yields
\begin{equation} \label{eq:energy_spectral_lin}
 \int_0^R (w_{nl}')^2 r \,dr + \int_0^R \Upsilon_{n} w_{nl}^2 r\,dr =\sigma_{nl}  \int_0^R  w_{nl}^2 r\,dr
\end{equation}
where we introduced the abbreviation
\begin{equation*} 
%\label{eq:def_Kcal_nl}
\Upsilon_{n} = 1+\frac{n^2}{r^2} - \Lambda e^{\Phi} .
\end{equation*}

%We plug the product representation $w_{nl}=\Theta \omega_{nl,\delta}$ into \eqref{eq:energy_spectral_lin} to obtain
We represent $w_{nl}$ as $\Theta_\delta \tilde{w}_{nl,\delta}$ and multiplying \eqref{eq:Theta_n} by $\Theta_\delta^2 \tilde{w}_{nl,\delta}^2 r$, integrate from $0$ to $R$. Integrating by parts in the  first term we get
\begin{equation}\label{eq:Theta_energy}
  \int_0^R r\Theta_\delta'(\Theta_\delta \tilde{w}_{nl,\delta}^2)' \,dr
  + \int_0^R  \Upsilon_{n}( w_{nl,\delta}^2 - \delta \Theta_\delta \tilde{w}_{nl,\delta}^2) r \,dr
  = -\int_0^R \frac{n^2-1}{r}\Phi^\prime\Theta_\delta \tilde{w}_{nl,\delta}^2 \,dr.
\end{equation}
Subtracting \eqref{eq:Theta_energy} from \eqref{eq:energy_spectral_lin}, we find 
\begin{equation}
\label{factor_eigen_f_Liouville_lin}
\sigma_{nl}  \int_0^R  w_{nl}^2 r\,dr= \int_0^R(\Theta_\delta \tilde{w}_{nl,\delta}^\prime)^2\, rdr
                  +  \int_0^R \left( \Upsilon_{n} \delta 
                  -  \frac{n^2-1}{r^2}\Phi^\prime\right)\Theta_\delta \tilde{w}_{nl,\delta}^2 \, rdr.
\end{equation}
Now pass to the limit in this equality as $\delta \to 0$.  Observing that the $\liminf$ as $\delta\to +0$ of the last term in \eqref{factor_eigen_f_Liouville_lin} is nonnegative we obtain that $\sigma_{nl}\geq 0$ and if $\sigma_{ln}=0$,
then $w_{nl}= -\gamma \Phi^\prime(r) $, where $\gamma$ is a constant. In the latter case 
$w_{nl}(R)\not=0$, contradiction. Thus $\sigma_{nl}> 0$.
%\begin{align}
% \int_0^R \sigma_{nl} \underbrace{\Theta_\delta^2\tilde{w}_{nl,\delta}^2 r }\,dr
% &~= \int_0^R \left( \underbrace{(\Theta_\delta \tilde{w}_{nl,\delta}')^2 }
%                  + \left( \mathcal{K}_{nl} \delta 
%                  -  \frac{n^2-1}{r^2}\Phi_0'\right)\underbrace{\Theta_\delta \tilde{w}_{nl,\delta}^2}\right) r\,dr,\\
% \geq 0~~~~~~~ &  ~~~~~~~~~~~~~~~~~\geq 0 ~~~~~~~~~~~~~~~~~~~~~~~~~~~~~~~~~~~~~\geq 0 \notag
%\end{align}
%and since the eigenfunction $w_{nl}$ does not vanish identically, $\sigma_{nl}$ is necessarily positive if 
%\begin{equation}
% \mathcal{K}_{nl} \delta - \frac{n^2-1}{r^2}\Phi_0'  \geq 0
%\end{equation}
%which for any $n\geq 2$ and sufficiently small $\delta$ can always be achieved, no matter what $\mathcal{K}_{nl}$ is. 
\end{proof}

\begin{cor}
	\label{cor_exist_uniqu_LineraLiouv} For each $f\in H^{1/2}(\partial B_R)$ satisfying
		$$
		\int_{-\pi}^{\pi} f(R,\varphi)\, d\varphi=0,
		$$
		the problem
		\begin{equation}
		\label{NonHomDirich_LiouvLinearize}
		-\Delta g+g-\Lambda e^\Phi g=0 \quad\text{in}\ B_R,\quad g=f\quad\text{on}\ \partial B_R   
		\end{equation}
		has a solution. Moreover  precisely one such a solution  is orthogonal in $L^2(B_R)$ to all radially symmetric functions $w(r)$.   
\end{cor}

\begin{proof}
	Introduce the solution $\tilde g$ of
	\begin{equation}
	\label{NonHomDirich_Laplace}
	-\Delta \tilde g=0 \quad\text{in}\ B_R,\quad \tilde g=f\quad\text{on}\ \partial B_R,  
	\end{equation}	
and observe that $\tilde g=\sum_{n=1}^\infty r^n(a_n \cos n\varphi+b_n\sin n\varphi)$. Then 
a solution of the problem 
$$
-\Delta (g-\tilde g)+(g-\tilde g)-\Lambda e^\Phi (g-\tilde g)=\Lambda e^\Phi\tilde g-\tilde g  \quad\text{in}\ B_R,\quad g-\tilde g=0\quad\text{on}\ \partial B_R   
$$
is obtained by separation of variables and applying Proposition \ref{Spectrum_Liouville_linaerized}.
\end{proof}

\section{Existence of solutions of the %Liouville type 
	problem  \eqref{liouvillewithv}-\eqref{paraconditions1}}
	\label{sec:solution_dirichlet_with_v}	
%Dirichlet problem \eqref{liouvillewithv}-\eqref{paraconditions1}} 
%beyond radial steady states

For a given $R>0$ we consider a fixed steady state $(\Lambda_0,\Phi_0)\in \mathcal{A}$. 
Using well-established techniques based on the Implicit Function Theorem, see, e.g., Chapter I in
\cite{Kor2012}, we construct a family of solutions  of \eqref{liouvillewithv}-\eqref{paraconditions1}  in domains $\Omega=\Omega_\eta$  given by  
\begin{equation}
 \Omega_\eta = \{(x,y)=r(\cos\varphi,\sin\varphi)\, |\, 0\leq r < R + \eta(\varphi), \, -\pi\leq\varphi<\pi \}  \label{eq_domain_rho}
\end{equation}
with sufficiently small $\eta\in C^{2,\gamma}(\mathbb{S}^1)$, $0<\gamma<1$, and with small, but not necessarily zero, velocity $V$. Hereafter, slightly abusing the notation, we identify the angle $\varphi\in [-\pi,\pi)$ with the corresponding point $(\cos\varphi,\sin\varphi)$ on the
unit circle $\mathbb{S}^1$. 

In order to reduce the construction to a fixed domain  
we introduce the mapping $Q_\eta:\Omega_\eta\to B_R$ defined in polar coordinates by
\begin{equation}
\label{standard_diffeomorfisme}
(r,\varphi)\mapsto Q_\eta(r,\varphi):=(r-\chi(r)\eta(\varphi),\varphi) 
\end{equation} 
where $\chi\in C^\infty(\mathbb{R}) $ is such that $\chi(r)=0$ when $r<R/3$ and $\chi(r)=1$ when 
$r>R/2$. Clearly, \eqref{standard_diffeomorfisme} defines a $C^2$-diffeomorphism whenever $\eta$ is sufficiently small together with its first and second derivatives. 

Among all perturbations $\Omega_\eta$ we single out those satisfying the area preservation condition
\begin{equation} 
\label{TotalVolPrezerv}
\frac{1}{2}\int_{-\pi}^{\pi} (R+\eta)^2\,d\varphi  =\pi R^2,
\end{equation} 
or in linear approximation
\begin{equation*} 
%\label{infinitisimalVolPrezerv}
\int_{-\pi}^{\pi} \eta(\varphi)\,d\varphi =0. 
\end{equation*} 

%In this section we consider the family of steady states $(\Lambda,\Phi)\in \mathcal{K}_1$ in a neighborhood of a pair $(\Lambda_0,\Phi_0)$ satisfying the necessary bifurcation condition 
%\eqref{bifurccondition1}. We assume that steady states are parametrized by real numbers 
%$s\in [-1,1]$, i.e. $\Lambda=\Lambda_s$ and $\Phi=\Phi_s$,  so that the maps $s\mapsto \Lambda_s\in \mathbb{R}$ and $s\mapsto\Phi_s\in C([0,R])$ are analytic. We construct a family of solutions  $\Lambda=\Lambda(s,V,\rho)$ and $S=S(x,y,s,V,\rho)$ of the equation 
%\eqref{liouvillewithv} (with prescribed zero boundary data) in domains 
%%\begin{equation}
%%\Omega = \Omega_\rho = \{(x,y)=r(\cos\varphi,\sin\varphi)\, |\, 0\leq r < R + \rho(\varphi), \, -\pi\leq\varphi<\pi \},   \label{eq_domain_rho}
%%\end{equation}
%%equal to zero on $\partial \Omega_\rho$,
%for all $s\in [-1,1]$ and sufficiently small $V$ and $\rho$. 

%(in $C^{2,\gamma}([-\pi,\pi])$, $0<\gamma<1$).

The following proposition establishes existence of solutions of problem \eqref{liouvillewithv}-\eqref{paraconditions1}.
 %in any given domain $\Omega_\eta$ differing not too much from a disk. 
 These solutions are obtained as perturbations of the radially symmetric steady states from Section \ref{sec:steady_states}.

\begin{prop} 
	\label{DefineSolLiouvilNonrad}
There exists some $\ve>0$ such that for all $(V,\eta,z)\in \mathbb{R}\times C^{2,\gamma}(\mathbb{S}^1)\times \mathbb{R}$ in $\ve$-neighborhood $U_\ve$ of $0$ the problem \eqref{liouvillewithv}-\eqref{paraconditions1} admits a solution $\Lambda=\Lambda(V,\eta,z)$,  $S=S(x,y,V,\eta,z)$  in the domain $\Omega=\Omega_\eta$ (given by \eqref{eq_domain_rho}). Here $z$ is an auxiliary real parameter (to be specified in the proof) such that 
\begin{equation}
\label{AnalyticMap}
z\mapsto (\Lambda(0,0,z), S(\,\cdot\,,\,\cdot\,,0,0,z))\in \mathcal {A}_1\quad \text{for $|z|<\ve$} 
\end{equation}
defines an analytic parametrization of the curve  $\mathcal {A}_1$ in a neighborhood of $(\Lambda_0,\Phi_0)$.
%
%$|V|< \ve$, $\|\eta\|_{C^{2,\gamma}(\mathbb{S}^1)}< \ve$, $|s|< \ve$, where $z$ 
%is an auxiliary real parameter. 
Moreover, the mappings  
$$
(V,\eta,z)\mapsto \Lambda(V,\eta,z),  \quad (V,\eta,z)\mapsto P(\,\cdot\,,V,\eta,z):=
\frac{\partial S}{\partial \nu}(Q^{-1}_\eta(R\,\cdot\,),V,\eta,z)\bigl|_{\partial B_R}\bigr.%\in C(\partial B_R)
$$
belong to $C^1(U_\ve;\mathbb{R})$ and $C^1(U_\ve;C^{1,\gamma}(\mathbb{S}^1))$, respectively.  The derivatives $\partial_V\Lambda$  and 
$\partial_V P$ at $(0,0,z)=0$ are given by 
\begin{equation}
\label{partialLambda}
\partial_V\Lambda=0, %\quad \text{and}
\quad \partial_V P=\frac{\partial\phi_1}{\partial \nu}, %\quad \text{and}
\end{equation}
where $\phi_1$ is a unique, as in Corollary  \ref{cor_exist_uniqu_LineraLiouv}, solution of 
%\textcolor{red}{the unique solutions} 
\begin{equation}
\label{phi_1_Liouv}
-\Delta \phi_1+\phi_1=\Lambda(z) e^{\Phi(r,z)}(\phi_1 -r\cos\varphi) \quad \text{in}\ B_R,
\quad \phi_1=0 \quad \text{on}\ \partial B_R,
\end{equation}
with $\Lambda(z):=\Lambda(0,0,z)$, and ${\Phi(r,z)}:=S(x,y,0,0,z)$.
The derivatives 
$\partial_\eta \Lambda$ and  $\partial_\eta P$ at $(0,0,z)$ satisfy
\begin{equation}
\label{partialPpppp}
\langle 
\partial_\eta\Lambda, \rho\rangle=0,
\quad \ \langle 
\partial_\eta P, \rho\rangle=\partial^2_{rr}\Phi(R,z)\rho + \frac{\partial\phi_2}{\partial \nu} 
\end{equation}
for $\rho$ such that $\int_{-\pi}^{\pi} \rho(\varphi)\,d\varphi =0$,  
where  $\phi_2$ is a unique, as in Corollary  \ref{cor_exist_uniqu_LineraLiouv}, solution of 
%\textcolor{red}{the unique solutions} 
of the problem
\begin{equation}
\label{phi_2_Liouv}
-\Delta \phi_2+\phi_2=\Lambda(z) e^{\Phi(r,z)}\phi_2  \quad \text{in}\ B_R,
\quad \phi_2=-\Phi_0^\prime(R)\rho\ \quad \text{on}\ \partial B_R.
% \quad \text{on}\ \partial B_R,
\end{equation}

%
%{$\langle \Lambda_\eta, s\rangle $ of $\Lambda(z,V,\eta)$ at zero are }
%
% 
%
%	
%%	
%	 Moreover, $\Lambda(s,V,\eta)=\Lambda_s$,
%	$S(x,y,s,0,0)=\Phi_s(x,y)$ and  the dependence on the parameters $V$, $\eta$ is smooth in the sense that 
\end{prop}	 

\begin{proof}
	Using the diffeomorphism $Q_\eta$, equation \eqref{liouvillewithv} in terms of $\tilde S=S\circ Q_\eta^{-1}$ (recall that  $Q_\eta$  is defined by \eqref{standard_diffeomorfisme}) reads
	%the equation \eqref{liouvillewithv} rewrites in terms of 
	%$\tilde S=S\circ Q_\eta^{-1}$ as
%	\begin{equation}
%	-\Delta \tilde S   +\tilde S -\Lambda e^{\tilde S-VQ_x^{-1}(x,y)}=\partial_x \left(D_{11} \partial_x\tilde S+D_{12}\partial_y\tilde S\right)+
%	\partial_y \left(D_{21}\partial_x\tilde S+D_{22}\partial_y\tilde S\right)
%	\end{equation}
% 	\begin{equation}
% 	\label{uzhas_kakoi}
% 	\begin{aligned}
% 	F:=-\Delta\tilde S   +\tilde S	 -
% 	\Lambda e^{\tilde S-V\tilde r\cos\varphi}-& \left((\chi^\prime \eta)^2-2\chi^\prime \eta+(\chi \eta^\prime)^2/\tilde{r}^2
% 	\right)\tilde S_{r r} \\
% 	&
% 	+\left(1/r-1/\tilde{r}+\chi^\prime \eta/\tilde{r}+\chi^{\prime\prime} \eta
% 	+\chi \eta^{\prime\prime}/\tilde{r}^2
% 	\right)\tilde S_{r}\\&
% +\chi \eta^\prime \tilde S_{r \varphi}/\tilde{r}^2+ \tilde S_{\varphi\varphi }(1/r^2-1/\tilde{r}^2)=0, \quad\quad 0\leq r< R,
% \end{aligned}
% \end{equation}
\begin{align}
	F(\Lambda,\tilde S,V,\rho,z):=-\Delta\tilde S   +\tilde S-&
	\Lambda e^{\tilde S-V\tilde r\cos\varphi}+\left((\chi^\prime \eta)^2-2\chi^\prime \eta+(\chi \eta^\prime)^2/\tilde{r}^2
	\right)\tilde S_{r r}\notag \\
	+& 
	\left(1/r-1/\tilde{r}+\chi^\prime \eta/\tilde{r}+\chi^{\prime\prime} \eta
	+\chi \eta^{\prime\prime}/\tilde{r}^2
	\right)\tilde S_{r} \notag \\
	 &+ \chi \eta^\prime \tilde S_{r \varphi}/\tilde{r}^2+ \tilde S_{\varphi\varphi }(1/r^2-1/\tilde{r}^2)=0, \quad\quad 0\leq r< R, \label{uzhas_kakoi}
\end{align}
where $\tilde{r}=|Q^{-1}_\eta(r\cos\varphi,r\sin\varphi)|$. The operator 
$$
F: \,\mathbb{R}\times  C^{2,\gamma}(B_R)\cap C_0(\overline{B}_R) \times \mathbb{R} \times C^{2,\gamma}(\mathbb{S}^1) \times \mathbb{R}\ni (\Lambda,\, \tilde S,\, V,\, \eta, z) \, \mapsto\, F(\Lambda,\tilde S,V,\eta,z)\in C^{0,\gamma}(B_R),
$$ 
is continuously Fr\'echet differentiable with respect to $\tilde S$ in some neighborhood of $(\Lambda_0,\Phi_0,0,0,0)$, and the derivative $\partial_{\tilde S} F$ at the given steady state takes the form
$$
\langle\partial_{\tilde S} F(\Lambda_0,\Phi_0,0,0),  w \rangle =
-\Delta w  +w	 -
\Lambda_0 e^{\Phi_0}w.
$$
That means, if the  problem 
\begin{equation}
\label{LiNeArIzEd_bif}
-\Delta w  +w -
\Lambda_0 e^{\Phi_0}w=0	\quad \text{in}\ B_R,
\quad
w=0\quad \text{on}\ \partial B_R
\end{equation}
has only the trivial solution $w=0$, then $F_{\tilde S}(\Lambda_0,\Phi_0,0,0):C^{2,\gamma}(B_R)\cap C_0(\overline{B}_R)\to C^{0,\gamma}(B_R)$ is an isomorphism and by the Implicit Function Theorem, equation  \eqref{uzhas_kakoi} can be solved for $\tilde S$ by a continuous mapping
$(V,\rho,z)\mapsto \tilde S(\, \cdot\,,\,\cdot\,,V,\rho,z)$ in a neighborhood of $(\Lambda_0, 0,0)$, where we defined the parameter $z$ by setting $z:=\Lambda-\Lambda_0$ (equivalently providing $\Lambda(z) = \Lambda_0+z$). 

In the case when \eqref{LiNeArIzEd_bif} has a nonzero solution $w$ we know from the proof of Theorem \ref{firstExistTH} that there are no other linear independent solutions and $w$ satisfies the non-degeneracy condition
\begin{equation}
\label{nondegAGAIN} 
\int_{B_R} e^{\Phi_0}\, w dxdy\not= 0. 
\end{equation} 
We seek $\tilde S$ in the form $\tilde S=\Phi_0+zw +\phi$ with a new  unknown  $\phi$ 
orthogonal (in $L^2(B_R)$) to $w$, i.e. 
$$
\phi\in Y=\left\{ \phi\in C^{2,\gamma}(B_R)\cap C_0(\overline{B}_R)\,\,|\, \int_{B_R} \phi w \,dxdy= 0  \right\}.
$$
Then problem \eqref{uzhas_kakoi} rewrites as $G(\Lambda, \phi, V,\eta,z):= F(\Lambda, \Phi_0+zw +\phi , V, \eta,z)=0$.  We consider $z$ as well as $V$ and $\rho$ as parameters, and note that the operator
$$
G: \, \mathbb{R}\times  Y\ni (\Lambda, \phi)\, \mapsto\, 
G(\Lambda, \phi, V, \eta, z)\in C^{0, \gamma}(B_R).
$$ 
has a continuous  Fr\'echet derivative $\partial_{(\Lambda,\phi)}G$ 
and its value at $(\Lambda_0,0,0,0,0)=:p_0$ is given by  
$$
\langle \partial_{(\Lambda,\phi)}G(p_0),  (\zeta,w)\rangle =
-\Delta w +w	 -
\Lambda_0 e^{\Phi_0}w-\zeta e^{\Phi_0}.
$$
We claim that  $\partial_{(\Lambda,\phi)}G(p_0)$ is a one-to-one mapping of 
$\mathbb{R}\times Y$ onto $C^{0,\gamma}(B_R)$. Indeed, given $f\in C^{0,\gamma}(B_R)$, there exists a unique solution $w\in Y$ of the problem 
\begin{equation}
\label{FredHolmmm}
-\Delta w +w	 -
\Lambda_0 e^{\Phi_0}w-\zeta e^{\Phi_0}=f\quad \text{in}\ B_R,\quad
w=0\quad \text{on}\ \partial B_R
\end{equation}
if and only if $\zeta =- \int_{B_R} f w\, dxdy / \int_{B_R} e^{\Phi_0} w\, dxdy$, i.e. for every $f\in C^{0,\gamma}(B_R)$ there is a unique pair $(\zeta,v)\in\mathbb{R}\times Y $ such that \eqref{FredHolmmm} holds. Also, both the operator $\partial_{(\Lambda,\phi)}G(p_0)$ and its inverse $(\partial_{(\Lambda,\phi)}G(p_0))^{-1}$ are continuous:  for $\partial_{(\Lambda,\phi)}G(p_0)$ this fact is obvious while the continuity of $(\partial_{(\Lambda,\phi)}G(p_0))^{-1}$ follows by classical elliptic estimates (see, e.g. \cite{GilTru2001}). Thus we can apply the Implicit Function Theorem to establish existence of 
$\Lambda(z,V.\eta)$ and $\tilde S(\,\cdot\,,\,\cdot\,, z,V,\eta)$.

To prove \eqref{AnalyticMap} we can complexify the construction by allowing 
$z$ take complex values $z\in\mathbb{C}$. Then calculating the derivative $\partial/\partial \overline{z}$ of \eqref{uzhas_kakoi} at $(0,0,z)$ we obtain that $h:=\partial_{\overline z} \tilde S$ solves 
\begin{equation}
\label{LiNeArIzEn_foranalyticity}
-\Delta h  +h -
\Lambda e^{\Phi(r,z)}h=\partial_{\overline z}\Lambda \, e^{\Phi(r,z)}	\quad \text{in}\ B_R,
\quad
h=0\quad \text{on}\ \partial B_R,
\end{equation}
where $\Lambda=\Lambda(0,0,z)$ and $\Phi(r,z)=\tilde S(x,y,0,0,z)$.
Recall that if \eqref{LiNeArIzEd_bif} nas no nontrivial solutions, then $\Lambda=\Lambda_0+z$. Hence $\partial_{\overline z}\Lambda=0$ which in turn implies that $h=0$ for sufficiently small $|z|$. Now assume that there is a nontrivial solution $w$ of \eqref{LiNeArIzEd_bif} satisfying 
\eqref{nondegAGAIN} and assume that either $h\not =0$ or $\zeta:=\partial_{\overline z}\Lambda\not =0$. Then we can normalize the pair $(\zeta,h)$ so that either $\zeta =1$ or  $\zeta=0$ and $\|h\|_{C^{2,\gamma}(B_R)}=1$. In the case $\zeta =1$ the function 
$h$ still satisfies the a priori bound $\|h\|_{C^{2,\gamma}(B_R)}\leq C$ for sufficiently small $|z|$ thanks to the fact that $h\in Y$. This allows one to pass to the limit as $|z|\to 0$ (along a subsequence), to get a nontrivial  pair $(\zeta,h)\in \mathbb{C}\times Y$ satisfying 
\begin{equation*}
%\label{LiNeArIzEn_foranalyticity}
-\Delta h  +h -
\Lambda e^{\Phi_0}h=\zeta e^{\Phi_0}	\quad \text{in}\ B_R,
\quad
h=0\quad \text{on}\ \partial B_R.
\end{equation*}	
This contradiction completes the proof of analyticity.

%
% recall that  if \eqref{LiNeArIzEd_bif} nas no nontrivial solutions, then $\Lambda=\Lambda_0+z$ so that  $\partial_{\overline z}\Lambda=0$. Using once more the fact %that \eqref{LiNeArIzEd_bif} has no nontrivial solutions we  conclude that %\eqref{LiNeArIzEn_foranalyticity} has only trivial solution $H=0$
%for sufficiently small $z$. \textcolor{red}{It is easy to show that 
%$\partial_{\overline z}\Lambda(0)=0$ but in a neighborhood it is more involved ..., }

To calculate the derivatives  $\partial_V\Lambda$ and $\partial_V P$ at $(0,0,z)$
we  linearize 	\eqref{uzhas_kakoi} in $V$ to find that  $H_1:=\partial_V \tilde S$ satisfies 
\begin{equation}
\label{Def_of_H_1}
-\Delta H_1  +H_1	 -
\Lambda e^{\Phi(r,z)}(H_1-r\cos\varphi)=\partial_V\Lambda \, e^{\Phi(r,z)}, \quad \text{in}\ B_R,
\quad
H_1=0\quad \text{on}\ \partial B_R.
\end{equation}
Subtract the solution $\phi_1$ of \eqref{phi_1_Liouv} to get the following problem for $\partial_V\Lambda$ and $\tilde H_1:=H_1-\phi_1$:
\begin{equation}
\label{Def_of_tildeH_1}
-\Delta \tilde H_1  + \tilde H_1	 -
\Lambda e^{\Phi(r,z)}\tilde H_1=\partial_V\Lambda \, e^{\Phi(r,z)}, \quad \text{in}\ B_R
\quad
H_1=0\quad \text{on}\ \partial B_R.
\end{equation}
Following exactly the same reasoning as for \eqref{LiNeArIzEn_foranalyticity}, problem 
\eqref{Def_of_tildeH_1} has only the zero solution for sufficiently small $|z|$ (note that $\phi_1$ is orthogonal in $L^2(B_R)$ to all radially symmetric functions $w(r)$).

%
%Next, if  \eqref{LiNeArIzEd_bif} has no nontrivial solutions, then  $\partial_V\Lambda=0$
%thanks to the choice of the parametrization, while if a nontrivial solution $w$ exists, then it is radially symmetric $w=w(r)$ so that \eqref{nondegAGAIN} and the solvability condition for \eqref{Def_of_H_1} yield $\partial_V\Lambda=0$. Thus, in both cases $H_1=\phi_1$.

Finally we calculate  $\langle \partial_\eta \Lambda,\rho\rangle$ and  $H_2:=\langle \partial_\eta  \tilde S,\eta\rangle$ at $(0,0,z)$.  Linearizing 	\eqref{uzhas_kakoi} in $\eta$ we find that  $H_2$ solves
	\begin{equation}
	\label{uzhas_kakoi_esche}
-\Delta H_2   +H_2	 -
\Lambda e^{\Phi}H_2+ 2\chi^\prime \rho\,\partial^2_{rr}\Phi
+\left(\chi \rho/r^2 +\chi^\prime \rho/r+\chi^{\prime\prime} \rho
+\chi \rho^{\prime\prime}/r^2
\right)\partial_{r}\Phi=\langle \partial_\eta\Lambda,\rho\rangle e^{\Phi}
	\end{equation}
in $B_R$ with the boundary condition $H_2=0$ on $\partial B_R$. Note that the auxiliary 
function 
$$
H_3(r,\varphi):= \chi(r) \rho(\varphi) \partial_r\Phi(r,z)+\phi_2(r,\varphi)
$$ 
satisfies 
\begin{equation}
\label{uzhas_kakoi_esche_i_esche}
-\Delta H_3   +H_3	 -
\Lambda e^{\Phi}H_3+ 2\chi^\prime \rho\,\partial^2_{rr}\Phi
+\left(\chi \rho/r^2 +\chi^\prime \rho/r+\chi^{\prime\prime} \rho
+\chi \rho^{\prime\prime}/r^2
\right)\partial_{r}\Phi=0\quad \text{in}\ B_R,
\end{equation}
therefore subtracting \eqref{uzhas_kakoi_esche_i_esche} from  \eqref{uzhas_kakoi_esche}
we find 
\begin{equation}
\label{sovsem_ne_uzhas}
-\Delta (H_2-H_3)   +(H_2-H_3)	 -
\Lambda e^{\Phi}(H_2-H_3)=\langle \partial_\eta\Lambda,\rho\rangle e^{\Phi}\quad \text{in}\ B_R, \quad H_2-H_3=0\quad \text{on}\ \partial B_R.
\end{equation}
This problem has only trivial solution for sufficiently small $|z|$, i.e. $\langle \partial_\eta\Lambda,\rho\rangle=0$ and  
$\frac{\partial}{\partial \nu} H_2= \rho \partial^2_{rr}\Phi(R,z)+ \frac{\partial}{\partial \nu} \phi_2$.
\end{proof}

\section{Bifurcation of traveling waves}
\label{sec:bifurcation_via_degree}
	
%\section{Fixed poind problem for traveling waves}

In this section we will show that at the potential bifurcation point found in Section \ref{sec:bifurcation_cond}, a bifurcation to traveling waves does take place. 
%The necessary bifurcation condition being satisfied, we now check that the sufficient
% condition of the degree jump principle -- the Leray-Schauder degree of the steady state  
% solution changing at the potential bifurcation point -- is satisfied.}

Let $(\Lambda_0, \Phi_0)\in \mathcal{A}_1$ be as in Corollary \ref{OchenHorCorrolaryBifurk}. 
According to Proposition  \eqref{DefineSolLiouvilNonrad} there is a family 
of solutions $\Lambda=\Lambda(V,\eta,z)$,  $S=S(x,y,V,\eta,z)$  
of  \eqref{liouvillewithv}-\eqref{paraconditions1} in the domains $\Omega=\Omega_\eta$ 
(given by \eqref{eq_domain_rho}).  These solutions are guaranteed to exist in a $\ve$-neighborhood ($\ve>0$)
of $(V,\eta,z)=(0,0,0)$  in the parameter space  $\mathbb{R}\times C^{2,\gamma}(\mathbb{S}^1)\times \mathbb{R}$ where they
continuously (actually, smoothly) 
depend on the parameters. 
%triples of  parameters 
%$(V,\eta,z)$  and are defined in a neighborhood of zero in the parameters space 
%$\mathbb{R}\times C^{2,\gamma}(\mathbb{S}^1)\times \mathbb{R}$. 
Thus for given $V\not =0$ the problem 
\eqref{liouvillewithv}-\eqref{paraconditions2} is reduced to finding 
$\rho$ such that $S|_{\eta=\rho}$ satisfies \eqref{paraconditions2} on 
$\partial \Omega=\partial \Omega_\rho$. The parameter 
$z$ now acts  a  bifurcation parameter. 

%Having defined solutions of  \eqref{liouvillewithv}-\eqref{paraconditions1} by Proposition \ref{DefineSolLiouvilNonrad} for small (but not necessary zero) $V$ in domains 
%$\Omega_\rho$ which are perturbations of the disk $B_R$, 

Next we rewrite the additional 
boundary condition \eqref{paraconditions2} as a fixed point problem for a compact
operator. Calculating the curvature $\kappa$ of $\partial \Omega_\rho$ and the normal vector $\nu$ in polar coordinates we 
have    
\begin{equation}
V\frac{(R+\rho)cos\varphi+\rho^\prime sin\varphi}{\sqrt{(\rho^\prime)^2+(R+\rho)^2}}= %\frac{\partial S}{\partial \nu}
P-\beta \frac{(R+\rho)^2+2(\rho^\prime)^2-(R+\rho)\rho^{\prime\prime}}
{\left((\rho^\prime)^2+(R+\rho)^2\right)^{3/2}}+\lambda,
\label{KinEtik1}
\end{equation}
where $P=P(\varphi, V,\rho,z)=\frac{\partial S}{\partial \nu}(Q_\rho^{-1}(R,\varphi), V,\rho,z)$
is defined in Proposition \ref{DefineSolLiouvilNonrad}. 
Introducing the notation $H:=\sqrt{(\rho^\prime)^2+(R+\rho)^2}$, rewrite \eqref{KinEtik1} as
\begin{equation*}
%\label{KinEtik2}
\frac{(R+\rho)\rho^{\prime\prime}-(\rho^\prime)^2}
{(\rho^\prime)^2+(R+\rho)^2}=\frac{1}{\beta}\Bigl(V(R+\rho)cos\varphi+V\rho^\prime sin\varphi-H\Bigl(P
%\frac{\partial S}{\partial \nu}
+\lambda\Bigr)\Bigr)+1,
\end{equation*}
or
\begin{equation}
\label{KinEtik3}
\left(\arctan \frac{\rho^{\prime}}
{R+\rho}\right)^\prime=\frac{1}{\beta}\Bigl(V(R+\rho)cos\varphi+V\rho^\prime sin\varphi-H\Bigl(P+\lambda\Bigr)\Bigr)+1.
\end{equation}
It follows that
\begin{equation}
\label{lambda_poschitali}
\lambda=\frac{1}{\int_{-\pi}^{\pi}H\, {\rm d}\varphi} \Bigl(\int_{-\pi}^{\pi}\left(V(R+\rho)cos\varphi+V\rho^\prime sin\varphi-H\,P\right) d\varphi+2\pi\beta\Bigr).
\end{equation}

To proceed further we impose three natural  conditions on $\Omega_\rho$. First,
we only  consider domains $\Omega_\rho$ symmetric with respect to $x$-axis (this is suggested by the symmetry of the problem, we assume that the motion occurs in the  direction of $x$-axis),  
that is we require  $\rho$ to be an {\it even} function $\rho$. Second, to avoid  translated 
(in $x$-direction) copies of  the solutions,  we fix the center of mass of $\Omega_\rho$ at 
 the origin:
\begin{equation}
\int_{\Omega_\rho} x\,dxdy=0, \quad\text{or in polar coordinates}\quad \frac{1}{3}\int_{-\pi}^{\pi}(R+\rho)^3\cos\varphi\, d\varphi=0.  
\label{centerofmass}
\end{equation}
Third, we impose the linearized counterpart of the area preservation condition \eqref{TotalVolPrezerv},
\begin{equation}
\int_{-\pi}^{\pi}\rho(\varphi)\, d\varphi=0. 
\label{linearizedareaprezervation}
\end{equation} 

From \eqref{KinEtik3}, taking into account the fact that $\rho^\prime(0)=0$ ($\rho$ is even) and \eqref{linearizedareaprezervation}, we get 
\begin{equation}
\label{FixedPointTW}
\rho = K(\rho, V;z)-\frac{1}{2\pi}\int_{-\pi}^{\pi} K(\rho, V;z)\,  d\varphi,
\end{equation}
where 
\begin{equation*}
K(\rho, V;z):= \int_0^\varphi (R+\rho) \tan\left(\psi_1 +
\frac{1}{\beta}\int_0^{\psi_1}\Bigl(V(R+\rho)cos\psi_2+V\rho^\prime sin\psi_2 -H\Bigl(P+\lambda\Bigr) \Bigr)\, d \psi_2\right)\, d\psi_1
\end{equation*}
with $\lambda$ given by \eqref{lambda_poschitali}.
Thus the traveling waves problem 
\eqref{liouvillewithv}-\eqref{paraconditions2} is reduced to the fixed point problem 
\eqref{FixedPointTW} in the space 
\begin{equation}
\label{ClassForFixedPointTW}
\rho \in\mathcal{H}=\left\{\rho\in C^{2,\gamma}(\mathbb{S}^{1})\,|\, \text{$\rho$ is even and satisfies \eqref{linearizedareaprezervation}}\right\}.
\end{equation}
The following Lemma shows that the operator in  the  right hand side of \eqref{FixedPointTW} maps $\mathcal{H}$ into itself.  
\begin{lem} We have 
\begin{equation}
\label{Oper_for_TW}
\left(K(\rho, V; z )-\frac{1}{2\pi}\int_{-\pi}^{\pi} K(\rho, V;z)\, d\varphi\right)
\in \mathcal{H}\quad \text{whenever}\  \rho \in \mathcal{H}.
\end{equation}
\end{lem}
%It is clear that the operator in \eqref{Oper_for_TW} is a compact mapping in $C^{2,\gamma}\setminus \mathbb{R}$. 
\begin{proof} 
	The only non-obvious fact is that the operator in the right hand side of  \eqref{Oper_for_TW} maps even function to even ones. This fact follows from the symmetry  of solutions of  \eqref{liouvillewithv}-\eqref{paraconditions1} with respect to $x$-axis in domains $\Omega=\Omega_\rho$ with the same symmetry. 
	%with respect to $x$-axis. 
	The latter property  is the consequence of the uniqueness 
of solutions $\Lambda$ and $S$ constructed in Proposition \ref{DefineSolLiouvilNonrad}, 
it also follows from general results \cite{GidNiNir1979} on symmetry of solutions of semilinear PDEs.
\end{proof}

We  also consider the velocity $V$ as unknown, supplementing \eqref{FixedPointTW} with 
the equation 
\begin{equation}
V=V+\frac{1}{3}\int_{-\pi}^{\pi}(R+\rho)^3\cos\varphi\, d\varphi,  
\label{volpreseration}
\end{equation}
which is obtained by adding \eqref{centerofmass} to the tautological equality $V=V$. Then 
we get the fixed point problem
\begin{equation}
\label{finalfixedP_forTW}
(\rho, V)=(\overline {K}_\rho(\rho, V;z),\overline{K}_V(\rho,V;z)) \quad \text{in}\ \mathcal{H} \times \mathbb{R},
\end{equation}
 where
$$
 \overline {K}_\rho(\rho, V;z)=K(\rho, V; z )-\frac{1}{2\pi}\int_{-\pi}^{\pi} K(\rho, V;z)\, d\varphi,\quad
 \overline{K}_V(\rho,V;z)=V+\frac{1}{3}\int_{-\pi}^{\pi}(R+\rho)^3\cos\varphi\, d\varphi.
$$
Note that $\overline K$ is a compact operator of the class $C^1$. This allows 
one to employ the Leray-Schauder degree theory to show existence 
of nontrivial solutions of \eqref{finalfixedP_forTW} bifurcating from the trivial
solution branch (represented by the curve of radially symmetric steady states). Specifically, traveling wave solutions are obtained as a new branch appearing at the bifurcation point corresponding to the parameter value $z=0$ where the local Leray-Schauder index 
jumps.

%The bifurcation is determined 
%by finding a value of the parameter $z$ at which 
 
Recall that the local 
Leray-Schauder index of $I-\overline K(\,\cdot\,;z)$ (where $I$ denotes the identity operator) at zero is defined by means of the 
linearized operator $\overline{L}(\,\cdot\,)$ of $\overline K(\,\cdot\,;z)$ by
$$
{\rm ind}_{LS}[I-\overline K(\,\cdot\,;z), 0]=(-1)^{N(z)},
$$
where $N(z)$ is the number of eigenvalues of  $\overline{L}(\,\cdot\,;z)$  contained in $(1,+\infty)$, counted with (algebraic)
multiplicities. 
%of the eigenvalues of  $\overline{L}(\,\cdot\,;z)$  contained in $(1,+\infty)$. 
The linearized operator $\overline{L}(\,\cdot\,;z)=({L_\rho}(\,\cdot\,;z),L_V(\,\cdot\,;z))$ is given by 
\begin{equation}
L_\rho(\rho, V;z)=\frac{R^2}{\beta}\int_0^\varphi \int_0^{\psi_1}
\Bigl(
V\cos\psi_2 - V \partial_V P(\psi_2,0,0,z)
-\langle \partial_{\eta}P(\psi_2,0,0,z),\rho\rangle-\frac{\beta \rho}{R^2}
\Bigr) 
\,{\rm d} \psi_2\,{\rm d}\psi_1 -C,
\label{LinearizOper_rho}
\end{equation}
\begin{equation}
L_V(\rho, V;z)=V+R^2\int_{-\pi}^{\pi} \rho \cos\varphi \, d\varphi,
\label{LinearizOper_V}
\end{equation}
where $C$ is 
%the constant which equals to 
the mean value of the first term 
in  \eqref{LinearizOper_rho}.

\begin{lem}
	\label{eigen_val_of_linearizedLiouv_with_ev} 
The eigenvalues of the linearized operator $\overline{L}(\,\cdot\,;z)$ 
are the pairs of eigenvalues  $E=E_{0,1}(z)$ solving the equation
\begin{equation}
\label{KvadratnoeUravnenieBifurk}
\frac{\pi}{R \Phi^\prime (R;z)}\int_0^R  \Phi^\prime(r;z) \, r^2dr-\pi=\frac{\beta (E-1)^2}{R^4} 
\end{equation}
and those given by 
\begin{equation}
\label{bifurkat_nontr_SW}
E_l(z)=\frac{1}{l^2}+\frac{R^2 h_l^\prime(R;z)}{\beta l^2}+ \frac{R^2\Phi^{\prime\prime}(R;z)}{\beta l^2},\quad l=2,3,\dots
\end{equation}
via solutions $h_l(r;z)$ of the problem \eqref{EscheKuchaSobstvZn}.
\end{lem}

\begin{proof} Consider an  eigenvalue $E$ corresponding to a eigenvector $(V,\rho)$ with 
	$V=1$. 
Then we have 
\begin{equation}
\label{FourirCoeffUfirstEF}
	\int_{-\pi}^{\pi} \rho \cos \varphi \, {\rm d}\varphi= (E-1)/R^2,
\end{equation}
Differentiate the equation $L_\rho(\rho,1;z)=E\rho$  twice with respect to  $\varphi$:
$$
 \cos\varphi -  \partial_V P(\varphi,0,0,z)
 -\langle \partial_{\eta}P(\varphi,0,0,z),\rho\rangle-\frac{\beta \rho}{R^2}=\frac{\beta E}{R^2}\rho^{\prime\prime}.
$$
Multiply this equation by 
	$\cos \varphi$ and integrate from $-\pi$ to $\pi$ to get
	$$
	\pi -\int_{-\pi}^{\pi} \left(\partial_V P(\varphi,0,0,z)
	+\langle \partial_{\eta}P(\varphi,0,0,z),\rho\rangle\right)\cos \varphi d\varphi=-\frac{\beta (E-1)^2}{R^4} 
	$$
Note that $\partial_V P(\varphi,0,0,z)$ and $\langle \partial_{\eta}P(\varphi,0,0,z),\rho\rangle$ are identified in  Proposition \eqref{DefineSolLiouvilNonrad} by means of problems \eqref{phi_1_Liouv} and \eqref{phi_2_Liouv}. We can calculate the integral on the left hand side multiplying 
\eqref{phi_1_Liouv} and \eqref{phi_2_Liouv} by $\Phi^\prime(r) r\cos \varphi$, and integrating over 
$B_R$:
$$
 \int_{-\pi}^{\pi} \left(\partial_V P(\varphi,0,0,z)
 +\langle \partial_{\eta}P(\varphi,0,0,z),\rho\rangle\right)\cos \varphi d\varphi =\frac{\pi}{R \Phi^\prime (R;z)}\int_0^R  \Phi^\prime(r;z) \, r^2dr.
$$ 
Thus solutions of  \eqref{KvadratnoeUravnenieBifurk} are eigenvalues corresponding to 
eigenvectors $(1,\rho_{0,1})$ with  $\rho_{0,1} =(E_{0,1}-1)\cos\varphi /(\pi R^2)$ (cf. \eqref{FourirCoeffUfirstEF}) if $E_{0,1}\not =1$. In the special case $E_{0,1}=1$, there is the only eigenvector $(1,0)$ and the adjoint vector  $(0,\cos\varphi/(\pi R^2))$.

%From the equation we have $\Phi^{\prime\prime} (R;z)=-\frac{1}{R}\Phi^\prime(R;z)-\Lambda(z)$

Other eigenvectors are $(0,\rho)$ with $\rho = \cos l\varphi$, $l=2,3,\dots$. 
To calculate the corresponding  eigenvalues 
%and   pairs By Fourier analysis we also find other eigenvectors whose $V$-component is zero %and 
%$\rho =\cos l\varphi$, $l=2,3,\dots$. 
%The corresponding eigenvalues $E_l$ are obtained in following way.
%That is,  
we seek solutions  
of  problem \eqref{phi_2_Liouv} in the form $h_l(r)\cos l\varphi$,  which results in   
\begin{equation}
\label{EscheKuchaSobstvZn}
-\frac{1}{r}(r h^\prime_l(r))^\prime +\left(\frac{l^2}{r^2}+1\right)h_l(r)=\Lambda(z) e^{\Phi(r;z)}h_l(r)\quad
0<r<R, \quad h_l(0)=0,\ h_l(R)=-\Phi^\prime(R;z).
\end{equation}	
Then  we identify $\langle \partial_{\eta}P(\varphi,0,0,z),\rho\rangle=h_l^\prime(r)\cos l\varphi$
with the help of Proposition \ref{DefineSolLiouvilNonrad}. 
Plugging these relations into the equations $L_\rho(\rho,0;z)=E\rho$ leads to the formula 
\eqref{bifurkat_nontr_SW} for the eigenvalues $E=E_l$.
%
%\eqref{EscheKuchaSobstvZn}
%$$
%-\frac{1}{r}(rh^\prime_l(r))^\prime +\frac{l^2}{r^2}=\Lambda(z) e^{\Phi(r;z)}h_l(r)\quad
%0<r<R, \quad h_l(0)=0,\ h_l(R)=-\Phi^\prime(R;z),   
%$$
%and then
%$$
%-\Phi^{\prime\prime}(R;z)-h^\prime_l(R)-\frac{\beta}{R^2}=-\frac{\beta E l^2}{R^2}, 
%$$
%or 
%$$
%E_l=\frac{1}{l^2}+\frac{R^2 h_l^\prime(R)}{\beta l^2}+ \frac{R^2\Phi^{\prime\prime}(R;z)}{\beta l^2}.
%$$
%
\end{proof}

Assume now that none of eigenvalues \eqref{bifurkat_nontr_SW}  is $1$ for $z=0$, $E_l\not =1$,
$l=2,3,\dots$,
i.e.
\begin{equation}
\label{exceptional_va_Liouv}
\beta\not = \beta_l,\quad \beta_l= \frac{R^2}{l^2-1}(h_l^\prime(R;0)+\Phi_0^{\prime\prime}(R)), 
\ l=2,3,\dots .
\end{equation}
It is not hard to show that the exceptional values $\beta_l$ form a sequence converging to zero. Moreover, the following result holds.

\begin{lem} 
	\label{lem:exceptional_beta}
	Eigenvalues \eqref{bifurkat_nontr_SW} have the following uniform 
	in $-\ve<z<\ve$, $l\geq 2$ and $\beta>0$ bound
	\begin{equation}
	\label{Unif_EV_bound_lin_Liouv_TW}
	E_l\leq {C}\left(\frac{1}{\beta l}+\frac{1}{l^2}\right).
	\end{equation}
\end{lem}

\begin{proof}
Consider functions $\tilde h_{l+l_0}=(r/R)^{l+l_0}$, which are solutions of
\begin{equation}
	\label{EscheKuchaSobstvZnaux}
	-\frac{1}{r}(r \tilde h^\prime_{l+l_0}(r))^\prime +\Bigl(\frac{l+l_0}{r}\Bigr)^2\tilde h_{l+l_0}(r)=0\quad
	0<r<R, \quad \tilde h^\prime_{l+l_0}(0)=0,\ \tilde h^\prime_{l+l_0}(R)=1.
	\end{equation}
For sufficiently large $l_0$ functions $h_l(r;z)$, being solutions of \eqref{EscheKuchaSobstvZn}, are all supersolutions of \eqref{EscheKuchaSobstvZnaux}, therefore $h_l(r)\geq -\Phi^\prime(R;z)\tilde h_{l+l_0}(r)$. 
This leads to the uniform bound \eqref{Unif_EV_bound_lin_Liouv_TW}.
\end{proof}

This Lemma implies that under the condition \eqref{exceptional_va_Liouv}  none of the eigenvalues   \eqref{bifurkat_nontr_SW} is equal to $1$ when  $-\ve_0\leq z\leq \ve_0$, for some $0<\ve_0< \ve$. 
On the other hand by Lemma \eqref{est_li_zhizn_na_Marse} in any neighborhood of $z=0$ there are $z$ such 
that $E_{0,1}(z)$ have nonzero imaginary part and there are $z$ such that both $E_{0,1}(z)$
are real and the smallest one, say $E_0(z)$, satisfies $E_0(z)<1$ while $E_1(z)>1$. This shows the jump of the local Leray-Schauder index through $z=0$ and yields the following theorem which is the main result of this work.

\begin{thm} 
	\label{bifurc_tw_thm_myaso}
	Let $(\Lambda_0, \Phi_0)\in \mathcal{A}_1$ be as in Corollary \ref{OchenHorCorrolaryBifurk}. Assume also that the parameter $\beta$  from  \eqref{tw_boundary} satisfies  
	 the inequality $\beta \neq \beta_l$ 
	 %in  \eqref{exceptional_va_Liouv} 
	where $\beta_l$ are defined in \eqref{exceptional_va_Liouv}  and $h_l(r)$ are solutions of \eqref{EscheKuchaSobstvZn} with $\Lambda=\Lambda_0$, $\Phi=\Phi_0$. Then there exists a family of solutions of % \eqref{volpreseration}
%and 
\eqref{finalfixedP_forTW}(traveling waves) with $V\not=0$ bifurcating from trivial solutions at $z=0$.
\end{thm}
\begin{rem}
By the construction above in this Section the problem \eqref{volpreseration}-\eqref{finalfixedP_forTW} is equivalent to the original problem \eqref{liouvillewithv}-\eqref{paraconditions2}, thus Theorem \ref{bifurc_tw_thm_myaso} and Lemma \ref{lem:exceptional_beta} yield Theorem \ref{thm:main}.
\end{rem}

\begin{rem}
	
The exceptional values $\beta=\beta_l $, $l=2,3,\dots$ correspond to bifurcations of non-radial steady states, see Section     \ref{sec:nonradial_steady_states}. It is conjectured 
that for exceptional $\beta=\beta_l$ the bifurcation to traveling waves and  to non-radial steady states  occurs simultaneously. Since the set of exceptional values has zero measure, this case is not further investigated here.  
\end{rem}
\begin{proof} We just make the above described arguments more precise and detailed. Let $\ve_0$ be such that  none of the eigenvalues   \eqref{bifurkat_nontr_SW} is equal to $1$ when  $-\ve_0\leq z\leq \ve_0$. By Corollary \ref{OchenHorCorrolaryBifurk} there are
$-\ve_0\leq z_{\pm}\leq \ve_0$ such that the left hand side of  \eqref{KvadratnoeUravnenieBifurk} is negative,
say at $z_{-}$, and it is positive at $z_{+}$. Since the linearized operators $\overline {L}(\,\cdot\,;z_{\pm})$ does not have the eigenvalue $1$ the Leray-Schauder degree
${\rm deg}_{LS}(I-\overline K(\,\cdot\,;z_{\pm}),\overline {U}_\delta, 0)$ is well defined for 
every $\delta$-neighborhood   
$$
\overline {U}_{\delta}=\{(V,\rho)\,|\, |V|<\delta,\,\|\rho\|_{C^{2,\gamma}(\mathbb{S}^1)}<\delta\}
$$ 
of zero in $\mathbb{R}\times C^{2,\gamma}(\mathbb{S}^1)$, $0<\delta<\ve_1$, for some $0<\ve_1<\ve_0/2$. Moreover,
$$
{\rm deg}_{LS}(I-\overline K(\,\cdot\,;z_{\pm}),\overline {U}_\delta, 0)={\rm ind}_{LS}[I-\overline K(\,\cdot\,;z_{\pm}), 0]=(-1)^{N(z_{\pm})},
$$ 
where $N(z_{\pm})$ is the number of eigenvalues of  $\overline{L}(\,\cdot\,;z_{\pm})$  contained in $(1,+\infty)$. Since the number of eigenvalues  \eqref{bifurkat_nontr_SW} contained in $(1,+\infty)$
coincides at $z_{-}$ and $z_{+}$ while for eigenvalues $E_{0,1}$ it differs by one, we conclude that
$$
{\rm deg}_{LS}(I-\overline K(\,\cdot\,;z_{-})
,\overline {U}_{\delta}, 0)\not= {\rm deg}_{LS}(I-\overline K(\,\cdot\,;z_{+}),\overline {U}_{\delta}, 0).
$$
It follows that for some $-\ve_0\leq z_{\ast}(\delta)\leq \ve_0$ the mapping $K(\,\cdot\,;z_{\ast})$ has a fixed
point $(V_\delta,\rho_\delta)$ on $\partial \overline{U}_{\delta}$. It remains to show that among these solutions there are true traveling waves. To this end we prove that $V_\delta=\pm \delta$
for sufficiently small $\delta>0$, arguing by contradiction. Assume that $\|\rho_\delta\|_{C^{2,\gamma}(\mathbb{S}^1)}=\delta$ and $|V_\delta|<\delta$ along a subsequence $\delta=\delta_n\to 0$. Then plug $V=V_\delta$ and $\rho=\rho_\delta$ in 
\eqref{finalfixedP_forTW}:
\begin{equation}
\label{finalfixedP_forTW_linearPart}
( V_\delta,\rho_\delta)=\overline{K}(V_\delta,\rho_\delta;z_{\ast}(\delta))=\overline {L}(V_\delta,\rho_\delta;z_{\ast}(\delta)) +O(\delta^2),
\end{equation}
%\eqref{KinEtik3}  and \eqref{lambda_poschitali}, 
divide the resulting identity by $\delta$ 
and pass to the limit as $\delta \to 0$. One obtains,  extracting a further subsequence (if necessary), 
$$
V_\delta/\delta\to V,\quad \text{and}\quad  \rho_\delta/\delta \to \rho\quad \text{strongly in} \  C^{2,\gamma}(\mathbb{S}^1), 
$$
and 
$$
(V,\rho)=\overline {L}(V,\rho;z_{\ast}), 
$$
with some $-\ve_0\leq z_{\ast}\leq \ve_0$. Thus $\overline {L}(\, \cdot\, ;z_{\ast})$ has the eigenvalue $1$ and a corresponding eigenvector $(V,\rho)$ with $\|\rho\|_{C^{2,\gamma}(\mathbb{S}^1)}=1$. But this contradicts the proof of 
Lemma \ref{eigen_val_of_linearizedLiouv_with_ev} (recall that $\ve_0$ is chosen so that none of the eigenvalues \eqref{bifurkat_nontr_SW} equals $1$). The Theorem is proved.
\end{proof}
% % % % % % % %Example
In a particular case when the bifurcation occurs from minimal solutions, which for example, takes place for $R\geq 4$ according to the proof of  Lemma \ref{est_li_zhizn_na_Marse}, case 2, we can  calculate  several terms of the asymptotic expansion of the traveling wave solutions in powers of the velocity $V$. Here we present the first  three terms  in the expansion of 
the function $\rho$ which determines the shape of the 
domain,
\begin{equation}
\rho=-V^2\frac{\tilde
	S_2(R)}{\Phi_0^\prime(R)}\cos 2\varphi-V^3\frac{\tilde
	S_3(R)}{\Phi_0^\prime(R)}\cos 3\varphi+\dots
\end{equation}
where $\tilde S_2$ solves 
\eqref{eq:secondorder_stress_bis}-\eqref{eq:secondorder_stress_bis_bc}, $\tilde S_3$ solves 
\eqref{eq:thirdorder_stress_bis}-\eqref{eq:thirdorder_stress_bis_bc} and $\tilde \phi$ is a solution of \eqref{equatttion81}-\eqref{equatttion81BC} with $\Lambda=\Lambda_0$ 
and $\Phi=\Phi_0$. 
\begin{figure}[h!]
%	\label{fig:figure_free_boundary}
%	{wrapfigure}
%	[18]{r}{5.5cm}
%	\vspace{-0.2 in}
\centering{	\includegraphics[scale=.4]{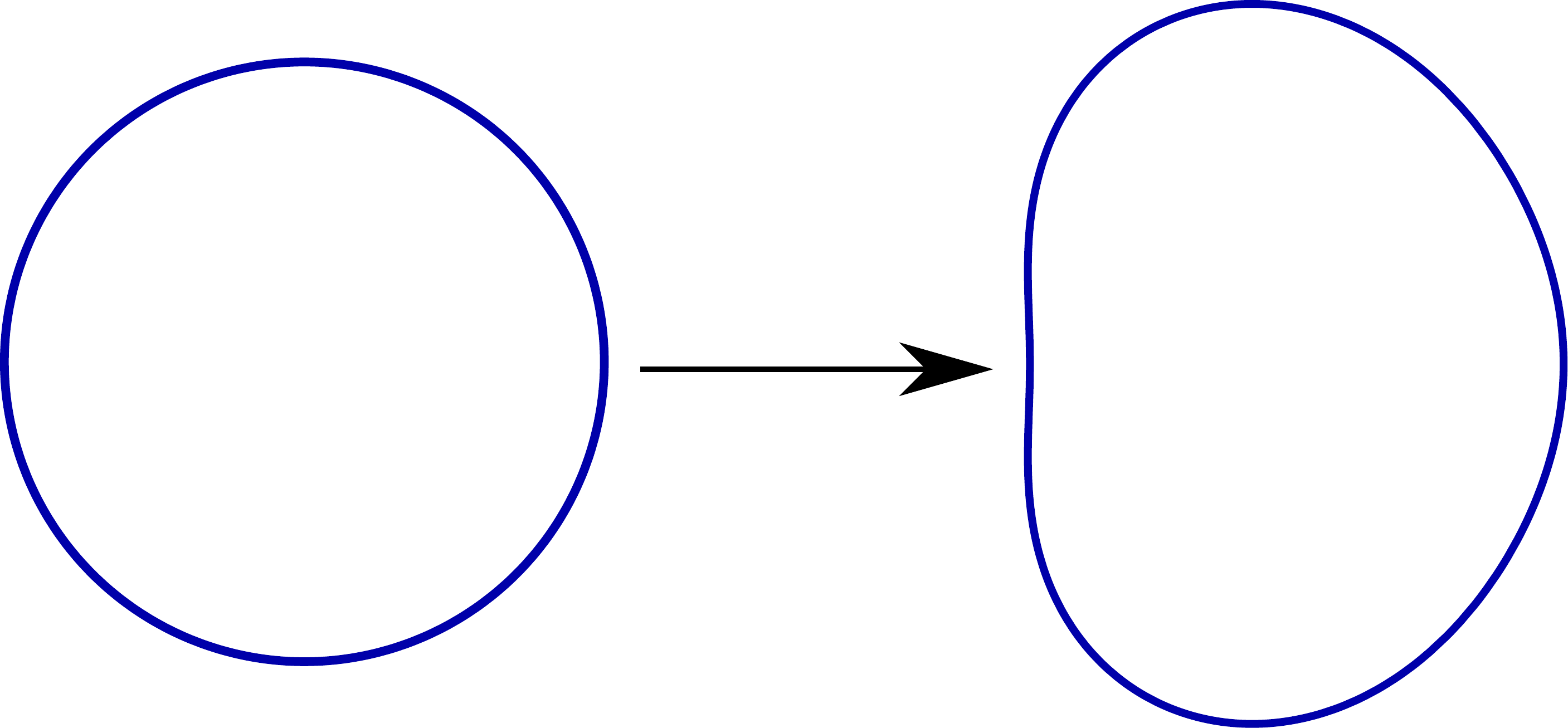}%.Shape6.pdf }
\caption{Approximate traveling wave shape with velocity $V=0.22$ bifurcated from a radial steady state with $R=4$, $\beta=5/8$.
	%Approximate traveling wave shape of a cell moving to the right with $R=4$, $\beta = 5/8$, $V=0.22$.
 The shape captures terms up to third order in $V$ computed as detailed in Appendix A.}
	\label{fig:figure_free_boundary}
} 
%For comparison the circular steady state shape of the resting cell is also shown (dashed)}}
\end{figure}
%{wrapfigure}

Fig. \ref{fig:figure_free_boundary} illustrates the  change in shape when the radially symmetric steady state bifurcates to a non-radial traveling wave.  The  calculations are presented  in the Appendix A.

%
%\begin{figure}[hbt!]
%%	\label{fig:figure_free_boundary}
%%	{wrapfigure}
%%	[18]{r}{5.5cm}
%%	\vspace{-0.2 in}
%\centering{	\includegraphics[width=0.6\linewidth]{TW_shape-med_adh_separate_no_text.pdf} \hskip2mm %
%                %\includegraphics[width=0.48\linewidth]{TW_shape_low_adh_separate_no_text.pdf}
%\caption{Approximate traveling wave shapes for $R = 1$% and $R=0.5$
%. For this value the bifurcation cannot be shown analystically to occur from the branch of minimal steady state solutions (however, simulations show that it does). The shape obtained for this smaller value of $R$ do not resemble moving keratocytes but rather other cell types as, e.g., neutrophils. Shapes of this kind have been reported both experimentally and in simulations, e.g., in \cite{VanFenEde2011}.}
%	\label{fig:other_shapes}
%} 
%%For comparison the circular steady state shape of the resting cell is also shown (dashed)}}
%\end{figure}
%
%In Fig. \ref{fig:other_shapes}, the approximate shape of the bifurcating traveling wave solutions are shown for another values of the scaled cell size $R$ for comparison. Obviously, the resulting shape strongly differs from the crescent like keratocyte shape and more closely resembles the shape of other cell types. This result indicates that the underlying mechanism of motion can be modeled in a very similar way for different types of cells. The differences in shape and speed of motion can in good approximation be captured by different physical parameters as density and strength of adhesions and total myosin activity.

% % % % % % % % % % % % % % % % % % % % % % % % % % %

\section{Nonradial steady states }
\label{sec:nonradial_steady_states}

While the main focus of this work is on traveling wave solutions, we also establish existence of steady state solutions lacking radial symmetry which, like traveling waves, form branches bifurcating from the family of radially symmetric steady states. Our analysis is restricted to bifurcation from pointwise minimal solutions of  \eqref{radialLiouville}-\eqref{radialLiouvilleBC}, whose existence is guaranteed by statement (ii)
of Theorem \eqref{firstExistTH}.  

As before we fix  $R>0$ and perform a local analysis in a neighborhood of a radially symmetric 
steady state $(\Lambda_0,\Phi_0)$. We assume that  $(\Lambda_0,\Phi_0)\in \mathcal{A}_1$, and moreover that $\Phi_0$ is a pointwise minimal solution of \eqref{radialLiouville}-\eqref{radialLiouvilleBC} for $\Lambda=\Lambda_0$. Therefore, by Proposition  \eqref{DefineSolLiouvilNonrad} there exists a family of solutions $\Lambda=\Lambda(V,\eta,z)$
$S=S(x,y,V,\eta,z)$ of the problems \eqref{liouvillewithv}-\eqref{paraconditions1} in domains
$\Omega_\eta$. The problem of finding solutions of \eqref{liouvillewithv}-\eqref{paraconditions2} with $V=0$ can be rewritten as the fixed point problem \eqref{finalfixedP_forTW}. Furthermore, in terms of the linearized operator 
$%\overline{L}
L_\rho(\,\cdot\,;z)$, given by \eqref{LinearizOper_rho}, the necessary condition for bifurcation of steady states at $(\Lambda_0,\Phi_0)$ is that $1$ is an eigenvalue of 
$%\overline{L}
L_\rho(\,\cdot\,;z)$ with $V=0$ and an eigenfunction $\rho$ satisfying the orthogonality condition $\int_{-\pi}^{\pi}\rho(\varphi)\cos\varphi \,d\varphi=0$. In view of Lemma \ref{eigen_val_of_linearizedLiouv_with_ev}, this necessary condition  can be reformulated as $E_l(0)=1$ for some $l=2,3,\dots$, 
where $E_l(z)$ are the eigenvalues given by \eqref{bifurkat_nontr_SW}.

\begin{lem}
	\label{lemma_pro_sobstvennye_zn_steady_st_liouv} Let $\Phi_0$ be a pointwise minimal
	solution of \eqref{radialLiouville}-\eqref{radialLiouvilleBC} with $\Lambda=\Lambda_0\geq 0$, and let $L_\rho(\,\cdot\,;z)$
be the family of linearized operators given by \eqref{LinearizOper_rho}, such that 
$z=0$ corresponds to the linearization around $(\Lambda_0,\Phi_0)$. Then  eigenvalues $E_l(z)$, $l=2,3,\dots$ of $L_\rho(\,\cdot\,;z)$, given by \eqref{bifurkat_nontr_SW}, are strictly increasing in $z$ for sufficiently small $z$, and if $E_{l_1}=E_{l_2}(0)=1$ for $l_1,l_2\geq 2$, then $l_1=l_2$. 
% and decreasing in $l$.
%
%
%have the following monotonicity property:
%\begin{equation}
%\frac{\beta l^2}{R^2} (E_l(z)-1/l^2)\quad \text{is strictly increasing in}\ z. 
%\end{equation}
%
%strictly increasing in $z$ and decreasing in $l$.	
\end{lem}  
  
\begin{proof}
	Rewrite problem \eqref{EscheKuchaSobstvZn},
	which determines $h_l(r;z)$, in terms of the new unknown $\psi_l(r;z):=h_l(r;z)+\Phi^\prime(r;z)$:
	\begin{equation}
	\label{EscheKuchaSobstvZnBIS}
	-\frac{1}{r}(r \psi^\prime_l(r))^\prime +\left(\frac{l^2}{r^2}+1-\Lambda(z) e^{\Phi(r;z)}\right)\psi_l(r)=\frac{l^2-1}{r^2}\Phi^\prime(r;z)\quad
	0<r<R, \quad \psi_l(0)=\psi_l(R)=0.
	\end{equation}	
Since $\Phi(r;z)$ are minimal solutions  of \eqref{radialLiouville}-\eqref{radialLiouvilleBC} for small  $z$, we can employ a comparison argument to prove that  $\psi_l(r;z_1)<\psi_l(r;z_2)$, $0<r<R$, whenever  $z_1>z_2$. Indeed, we have 
% \begin{equation}
% \label{EscheKuchaSobstvZnBIScomparison}
% \begin{aligned}
% -\frac{1}{r}\bigl(r (\psi^\prime_l(r;z_2)-\psi^\prime_l(r;z_1))\bigr)^\prime +&\left(\frac{l^2}{r^2}+1-\Lambda(z_2) e^{\Phi(r;z_2)}\right)(\psi_l(r;z_2)-\psi_l(r;z_1))\\
% =&
% \frac{l^2-1}{r^2}(\Phi^\prime(r;z_2)-\Phi^\prime(r;z_1))\\
% &+(\Lambda(z_2) e^{\Phi(r;z_2)}- \Lambda(z_1) e^{\Phi(r;z_1)})\psi_l(r;z_1).
% \end{aligned}
% \end{equation}
\begin{equation}
\label{EscheKuchaSobstvZnBIScomparison}
\begin{aligned}
&-\frac{1}{r}\bigl(r (\psi^\prime_l(r;z_2)-\psi^\prime_l(r;z_1))\bigr)^\prime +\left(\frac{l^2}{r^2}+1-\Lambda(z_2) e^{\Phi(r;z_2)}\right)(\psi_l(r;z_2)-\psi_l(r;z_1))\\
&=~ \frac{l^2-1}{r^2}(\Phi^\prime(r;z_2)-\Phi^\prime(r;z_1)) + (\Lambda(z_2) e^{\Phi(r;z_2)}- \Lambda(z_1) e^{\Phi(r;z_1)})\psi_l(r;z_1).
\end{aligned}
\end{equation}
Using factorization idea  as in Lemma \ref{Spectrum_Liouville_linaerized} we can show that
every solution of \eqref{EscheKuchaSobstvZnBIS} is negative in $(0,R)$, therefore the last term 
in \eqref{EscheKuchaSobstvZnBIScomparison} is positive. The same factorization 
trick applied to the equation 
\begin{equation*}
%\label{EscheKuchaSobstvZnBIScomparison}
\begin{aligned}
-\frac{1}{r}\bigl(r (\Phi^\prime (r;z_2)-\Phi^\prime(r;z_1))^\prime\bigr)^\prime +&\left(\frac{1}{r^2}+1-\Lambda(z_2) e^{\Phi(r;z_2)}\right)(\Phi^\prime(r;z_2)
-\Phi^\prime (r;z_1))\\
=&(\Lambda(z_2) e^{\Phi(r;z_2)}- \Lambda(z_1) e^{\Phi(r;z_1)})\Phi^\prime (r;z_1)
\end{aligned}
\end{equation*}		
%satisfied by  $\Phi^\prime(r;z_2)-\Phi^\prime(r;z_1)$ 
shows that $\Phi^\prime(r;z_2)-\Phi^\prime(r;z_1)> 0$ if $\Phi(r;z_1)>\Phi(r;z_2)$ on 
$(0,R)$ and $\Lambda(z_1)>\Lambda(z_2)$. Thus the right hand side of \eqref{EscheKuchaSobstvZnBIScomparison}
is positive  and the inequality $\psi_l(r;z_1)<\psi_l(r;z_2)$ follows. Moreover
the Hopf Lemma applied after a proper factorization (again as in Lemma \ref{Spectrum_Liouville_linaerized}) implies that $\psi_l^\prime(R;z_1)<\psi_l^\prime(R;z_2)$.
This proves monotonicity of $E_l(z)$.

To complete the proof of the Lemma assume by contradiction that $E_{l_1}(0)=E_{l_2}(0)$
for different $l_1,l_2\geq 2$, say $l_1>l_2$. Then by \eqref{bifurkat_nontr_SW} we have
\begin{equation}
\label{protiv_raznye_eigenval _liouv}
\psi_{l_1}^\prime(R;0)/(l_1^2-1)=\psi_{l_2}^\prime(R;0)/(l_2^2-1)=\beta/R^2.
\end{equation}
On the other hand functions $\psi_{l_i}^\prime(r;0)/(l_i^2-1)$, $i=1,2$ solve 
\begin{equation}
\label{Podelili_poluchili_liouvil}
-\frac{1}{r}(r \psi^\prime_{l_i}/(l_i^2-1))^\prime +\left(\frac{l^2_i}{r^2}+1-\Lambda_0 e^{\Phi_0}\right)\psi_{l_i}/(l_i^2-1)=\frac{1}{r^2}\Phi^\prime_0, \quad 0<r<R.
\end{equation}	
Then the pointwise inequalities  $0> \psi_{l_1}> \psi_{l_2}$ on $(0,R)$ follow,
and we have $\psi_{l_1}^\prime(R;0)/(l_1^2-1)<\psi_{l_2}^\prime(R;0)/(l_2^2-1)$, contradiction.
\end{proof}

The following theorem establishes the existence of bifurcations to not radially symmetric steady states if the surface tension parameter $\beta$ is sufficiently small.  
  
\begin{thm}
\label{thm:bifurc_teady_st}	
	 Given $R>0$, and $l=2,3,\dots$, for sufficiently small $\beta>0$ there is 
	a family of steady states solutions of  \eqref{tw_actinflow}-\eqref{tw_boundary}
whith the domian $\Omega$ whose boundary is given by 
\begin{equation}
\partial \Omega=\{(x,y)=(R+\rho_\delta(\varphi))(\cos\varphi,\sin\varphi)\,|\, -\pi\leq \varphi<\pi\}, 
\quad\text{where}\ \rho_\delta=\delta \cos l\varphi+o(\delta), 
\end{equation}
and $\delta>0$ is a small parameter.
\end{thm}

\begin{proof} The argument follows the line of Theorem \ref{bifurc_tw_thm_myaso}. 
The bifurcation condition \eqref{bifurccondition1} for traveling waves is now replaced by 
\begin{equation}
\label{steadystatesbifurccondition}
\frac{\psi_{l}^\prime(R;0)}{l^2-1}=\beta/R^2,
\end{equation}
where $\psi_l(r;0)$ is a solution of \eqref{EscheKuchaSobstvZnBIS} for $z=0$, 
and this latter condition is always satisfied at some pair 
$(\Lambda_0,\Phi_0)\in\mathcal{A}_0$ 
%some minimal solution $(\Lambda_0,\Phi_0)$
%of \eqref{radialLiouville}-\eqref{radialLiouvilleBC}
, provided $\beta>0$ is sufficiently 
small. Note that in contrast to \eqref{bifurccondition1} the condition \eqref{steadystatesbifurccondition} depends on $\beta$. Considering $\beta>0$ so small that the eigenvalues $E_{0,1}(z)$ (of
the linearized operator $\overline{L}(\,\cdot\,;z)$), given by 
\eqref{KvadratnoeUravnenieBifurk},  are bounded away from $1$, and using Lemma \ref{lemma_pro_sobstvennye_zn_steady_st_liouv} we see that for sufficiently small $z$  only the
eigenvalue $E_l(z)$ takes value $1$ and the sign of $E_l(z)-1$ changes. This allows us to establish the bifurcation of non-radial steady states analogously to Theorem \ref{bifurc_tw_thm_myaso}.% (about existence of traveling wave solutions). 
%
%Alternatively,  in the proof  Theorem \ref{thm:bifurc_teady_st} a bifurcation theorem from \cite{CraRab1971} (Teorem 1) can also be employed, thanks to the fact that \eqref{steadystatesbifurccondition} is satisfied on a minimal solution and therefore  the derivative  $\frac{d}{dz} \psi_{l}^\prime(R;z)\not=0$ at $z=0$. This latter fact yilds the condition (d) in Theorem \ref{thm:bifurc_teady_st}. 
\end{proof}

\section{Conclusions}

We consider a two dimensional Keller-Segel type elliptic-parabolic system with free boundary governed by a nonlocal kinematic condition which involves boundary curvature. This system models the motility of a eukaryotic cell on a flat substrate and is obtained as a reduction \cite{Mog_private} of the more complicated model from \cite{BarLeeAllTheMog2015}. We show that the model captures the main biological features of cell motility such as persistent motion and breaking of symmetry which have been studied in numerous experimental works, e.g., \cite{BarLeeAllTheMog2015,KerPinAllBarMarMogThe08}. In the model under consideration these two features correspond to bifurcation from radial steady states to non-radial steady states and traveling waves.
In particular, our  analytical and numerical  calculations capture emergence of asymmetric shapes of the traveling waves in this bifurcation, 
see Fig. \ref{fig:figure_free_boundary}. Specifically, the asymmetry of the cell shape 
depicted on Fig. \ref{fig:figure_free_boundary} qualitatively agrees 
with that of an actual moving cell as observed in  
%in numerical simulations qualitatively 
\cite{BarLeeKerMogThe2011}.

 The results are obtained by a two step procedure.
  First we reduce the problem of finding traveling waves/steady states to a Liouville type equation with an additional boundary condition due to the free boundary setting. Using methods from \cite{CraRab1971}  based on the Implicit Function Theorem, we further reduce the problem to a fixed point problem for a nonlinear compact mapping. Second, Leray-Schauder degree theory is applied to the this fixed point problem to prove existence of both  traveling waves and nonradial steady states.

{\em Acknowledgments.} Work of LB, JF, VR was partially supported by NSF grant DMS-1405769. The work of  VR  was  also partially supported by the PSU Center ``Mathematics of Living and Mimetic Matter". The authors are thankful to PSU students Matthew Mizuhara and Hai Chi for careful reading of the manuscript and  suggestions leading  to  its improvement. We also would like to acknowledge numerous  fruitful discussions with our colleagues A. Mogilner and L. Truskinovsky.

\bibliographystyle{plain}
\bibliography{references}

\begin{appendices}

\section{Asymptotic expansion of traveling waves near bifurcation point and  emergence of asymmetric shapes}
\label{sec:appendixA}

In this Appendix we  construct  several terms of the asymptotic expansion of the free boundary problem \eqref{liouvillewithv}-\eqref{paraconditions2}.    	
This is done for the case when the necessary bifurcation condition
\eqref{bifurccondition1}(Section \ref{sec:bifurcation_cond}) is
satisfied on a pair $(\Lambda_0, \Phi_0)$ with $\Phi_0$ being a
minimal solution of
\eqref{radialLiouville}-\eqref{radialLiouvilleBC}. Then the  bifurcating
traveling waves can be expanded in a (formal) series  
in a small parameter $\ve:=V$. This
expansion can be rigorously justified  using
Lyapunov-Schmidt reduction. While the first order approximation is already introduced in Section \ref{sec:bifurcation_cond},  here we calculate the first three terms in  this
asymptotic expansion and justify the assumption that the first order correction to $\Lambda_0$ is zero. Note that the first order correction to the shape of the domain is zero, the second order is symmetric with respect to the $y$-axis, and the asymmetry emerges in the third correction term. 

We seek the unknown domain $\Omega$ in the form $\Omega= \{(r \cos\varphi, r\sin\varphi) \mid \varphi\in[-\pi,\pi), 0\leq r < R
+\rho(\varphi)\}$ and  introduce the following expansions for the solutions of \eqref{liouvillewithv}-\eqref{paraconditions2}
% $V=\ve$:
\begin{equation*}
%\label{AsExpOmegaLiouville}
\rho=
\ve\rho_1 + \ve^2 \rho_2 +\ve^3\rho_3+O(\ve^4), \quad
S= \Phi_0(r) + \varepsilon S_1 + \varepsilon^2 S_2 + \ve^3
S_3+O(\ve^4),
\end{equation*}
\begin{equation*}
%\label{AsExpSLiouville} 
\Lambda=\Lambda_0+\ve
\Lambda_1+\ve^2\Lambda_2+\ve^3 \Lambda_3+O(\ve^4),\quad
\text{and}\quad \lambda = \lambda_0 + \ve \lambda_1 +
\ve^2\lambda_2 +\ve^3\lambda_3+O(\ve^4),
\end{equation*}
 where $\lambda_0= {\beta}/{R} -\Phi^\prime_0(R)$  follows  from  the leading term in the expansion of   \eqref{paraconditions2} in $\ve=V$. 
Plugging the above expansions into \eqref{liouvillewithv}-\eqref{paraconditions2} and equating the terms of order $\ve$, $\ve^2$, $\ve^3$  yields the following equations
\begin{equation}
\label{FirstOrderLiouvAppend} -\Delta S_1+S_1= \Lambda_0
e^{\Phi_0(r)}(S_1-x)+\Lambda_1 e^{\Phi_0(r)},%\quad \text{in}\ B_R,
\end{equation}
\begin{equation}
\label{SecondOrderLiouvAppend} -\Delta S_2+S_2= \Lambda_0
e^{\Phi_0(r)}S_2+
\frac{\Lambda_0}{2}
e^{\Phi_0(r)}(S_1-x)^2+\Lambda_1 e^{\Phi_0(r)}(S_1-x)+\Lambda_2
e^{\Phi_0(r)},%\quad \text{in}\ B_R,
\end{equation}
\begin{equation}
\begin{aligned}
-\Delta S_3+S_3  - \Lambda_0 e^{\Phi_0(r)} S_3 = & \Lambda_0 e^{\Phi_0(r)}\left(
(S_1-x)S_2 + (S_1- x)^3/6\right)\\
& +\Lambda_1  e^{\Phi_0(r)}\left(S_2+(S_1-x)^2/2\right)+
\Lambda_2 e^{\Phi_0(r)}(S_1-x) + \Lambda_3 e^{\Phi_0(r)}
\label{eq:thirdorder_stress_bisAppend}
\end{aligned}
\end{equation}
in $B_R$ with boundary conditions
\begin{equation}
\label{FirstOrderLiouv_1bc_Append}
S_1(R,\varphi)+\Phi_0^\prime(R)\rho_1(\varphi)=0
%\quad \text{on}\ \partial B_R,
\end{equation}
%\begin{equation}
%\label{SecondOrderLiouv_1bc_Append}
%S_2(R,\varphi)+\Phi_0^\prime(R)\rho_2(\varphi)=-\partial_r S_1(R,\varphi)\rho_1-\Phi_0^{\prime\prime}(R)\rho_1^2(\varphi)/2
%\quad \text{on}\ \partial B_R,
%\end{equation}
\begin{equation}
\label{SecondOrderLiouv_1bc_Append}
S_2(R,\varphi)+\Phi_0^\prime(R)\rho_2(\varphi)=T_1(\varphi)
%\quad \text{on}\ \partial B_R,
\end{equation}
%
%\begin{equation}
%\begin{aligned}
%\label{ThirdOrderLiouv_1bc_Append}
%S_3(R,\varphi)+\Phi_0^\prime(R)\rho_3(\varphi)
%=&-\partial_r S_2(R,\varphi)\rho_1-\partial_r S_1(R,\varphi)\rho_2
%\\
%&-\Phi_0^{\prime\prime\prime}(R)\rho_1^3(\varphi)/6
%-\partial^2_r S_1(R,\varphi)\rho_1^2(\varphi)/2-\Phi_0^{\prime\prime}(R)\rho_1(\varphi)\rho_2(\varphi)
%\end{aligned}
%\quad \text{on}\ \partial B_R,
%\end{equation}
\begin{equation}
\label{ThirdOrderLiouv_1bc_Append}
S_3(R,\varphi)+\Phi_0^\prime(R)\rho_3(\varphi) =-\partial_r
S_1(R,\varphi)\rho_2 +T_2(\varphi)
\end{equation}
and
\begin{equation}
\label{FirstOrderLiouv_2bc_Append} \cos\varphi=
\partial_r S_1(R,\varphi)+\Phi_0^{\prime\prime}(R)\rho_1(\varphi)
+\frac{\beta}{R^2}(\rho_1^{\prime\prime}(\varphi)+\rho_1(\varphi))
+\lambda_1,
\end{equation}
%
%\begin{equation}
%\begin{aligned}
%\label{SecondOrderLiouv_2bc_Append}
%\sin\varphi\frac{\rho_1^\prime(\varphi)}{R}=&\partial_r
%S_2(R,\varphi)+
%\Phi_0^{\prime\prime}(R)\rho_2(\varphi)+\frac{\beta}{R^2}(\rho^{\prime\prime}_2(\varphi)+\rho_2(\varphi))
%\\
%&+\Phi_0^{\prime\prime\prime}(R)\frac{\rho_1^2(\varphi)}{2}
%-\Phi_0^{\prime}(R)\frac{(\rho_1^\prime(\varphi))^2}{2R^2}+\partial^2_r
%S_1(R,\varphi)\rho_1(\varphi)-\partial_\varphi
%S_1(R,\varphi)\frac{\rho_1^\prime(\varphi)}{R^2}\\
%& -\frac{\beta}{R^3}
%(\rho_1^2(\varphi)+2\rho_1(\varphi)\rho_1^{\prime\prime}(\varphi)+(\rho_1^\prime(\varphi))^2/2),
%\end{aligned}
%\end{equation}
\begin{equation}
\label{SecondOrderLiouv_2bc_Append} 0=
\partial_r S_2(R,\varphi)+
\Phi_0^{\prime\prime}(R)\rho_2(\varphi)
+\frac{\beta}{R^2}(\rho^{\prime\prime}_2(\varphi)+\rho_2(\varphi))
+T_3(\varphi)+\lambda_2
\end{equation}
%
%\begin{equation}
%\begin{aligned}
%\label{ThirdOrderLiouv_2bc_Append}
%\sin\varphi\frac{\rho_2^\prime(\varphi)}{R}-\cos\varphi
%\frac{(\rho_1^\prime(\varphi))^2}{2R^2}=&\partial_r
%S_3(R,\varphi)+
%\Phi_0^{\prime\prime}(R)\rho_3(\varphi)+\frac{\beta}{R^2}(\rho^{\prime\prime}_3(\varphi)+\rho_3(\varphi))
%\\
%&+\Phi_0^{\prime\prime\prime}(R)\rho_1(\varphi)\rho_2(\varphi)
%+\Phi_0^{\prime\prime\prime\prime}(R)\frac{\rho_1^3(\varphi)}{6}
%-\Phi_0^{\prime}(R)\frac{(\rho_1^\prime(\varphi))^2}{2R^2}+\partial^2_r
%S_1(R,\varphi)\rho_1(\varphi)-\partial_\varphi
%S_1(R,\varphi)\frac{\rho_1^\prime(\varphi)}{R^2}\\
%& -\frac{\beta}{R^3}
%(\rho_1^2(\varphi)+2\rho_1(\varphi)\rho_1^{\prime\prime}(\varphi)+(\rho_1^\prime(\varphi))^2/2)
%\end{aligned}
%\end{equation}
\begin{equation}
\begin{aligned}
\label{ThirdOrderLiouv_2bc_Append}
\frac{1}{R}\rho_2^\prime(\varphi)\sin\varphi=&\partial_r
S_3(R,\varphi)
+
\Phi_0^{\prime\prime}(R)\rho_3(\varphi)
\\
&+\partial^2_r S_1(R,\varphi) \rho_2(\varphi)
-\partial_\varphi
S_1(R,\varphi)\frac{\rho_2^\prime(\varphi)}{R^2}+
\frac{\beta}{R^2}(\rho^{\prime\prime}_3(\varphi)+\rho_3(\varphi))
+T_4(\varphi)+\lambda_3,
\end{aligned}
\end{equation}
where $T_i$, $i=1,\dots 4$ denote various terms containing
factors $\rho_1(\varphi)$ or $\rho^\prime_1(\varphi)$ which will
be shown to vanish.

As explained in Section \eqref{sec:bifurcation_via_degree}, due to the symmetry of the problem  we only consider even functions $\rho$.
Moreover we impose the condition that the area of $\Omega$ is
equal to that of the disc $B_R$ and fix the center of mass of the
domain at the origin
% linearized counterpart \eqref{??} of the area
%preservation condition \eqref{??} and the orthogonality condition
%$\int_{-\pi}^\pi \rho(\varphi)\cos\varphi\, d\varphi=0$
to get rid of solutions obtained by infinitesimal shifts of the
domain. To the order $\ve$ these two conditions yield
\begin{equation}
\label{add_cond_Liouv_Appen}  
\int_{-\pi}^{\pi}\rho_1 \,d\varphi=0, \quad 
\int_{-\pi}^{\pi}\rho_1 \cos \varphi \,d\varphi =0.
\end{equation}
Since $\Phi_0$ is a minimal solution of
\eqref{radialLiouville}-\eqref{radialLiouvilleBC} we can locally
parametrize solutions $(\Lambda,\Phi(r,\Lambda))$ of
\eqref{radialLiouville}-\eqref{radialLiouvilleBC}  by $\Lambda$ so
that $\Phi_0(r)=\Phi(r,\Lambda_0)$. Expanding $\rho_1$ into a
Fourier series $\rho_1=\sum c_l\cos l\varphi$ we find from
\eqref{FirstOrderLiouvAppend},\eqref{FirstOrderLiouv_1bc_Append}
that
\begin{equation*}
%\label{repr_S1_Liouv_Append}
S_1=\tilde \phi(r,\Lambda_0)\cos\varphi+\Lambda_1\partial_\Lambda
\Phi(r,\Lambda_0)+\sum c_l h_l(r)\cos l\varphi,
\end{equation*}
where $\tilde \phi(r,\Lambda)$ are solutions of
\eqref{equatttion81}-\eqref{equatttion81BC} and $h_l$ are
solutions of problems \eqref{EscheKuchaSobstvZn} with
$\Lambda=\Lambda_0$ and $\Phi=\Phi_0$ (since $\Phi_0$ is a minimal
solution of \eqref{radialLiouville}-\eqref{radialLiouvilleBC},
solutions $h_l$ of \eqref{EscheKuchaSobstvZn} are defined uniquely). By
\eqref{add_cond_Liouv_Appen} the first Fourier coefficients satisfy $c_0=c_1=0$.
Moreover, assuming that the condition \eqref{exceptional_va_Liouv} is
satisfied we find by virtue of \eqref{FirstOrderLiouv_2bc_Append}
that all other Fourier coefficients $c_l$ are also zero, i.e.
$\rho_1=0$. Thus
\begin{equation}
\label{reprBis_S1_Liouv_Append} S_1=\tilde
\phi(r,\Lambda_0)\cos\varphi+\Lambda_1\partial_\Lambda
\Phi(r,\Lambda_0)
\end{equation}
(next we show that actually $\Lambda_1=0$).

%Now collecting terms of the order $\ve^2$ in
%\eqref{liouvillewithv}-\eqref{paraconditions2} (obtained after
%substituting the above formal series), we get

%with boundary conditions
%\begin{equation}
%\label{SecondOrderLiouv_1bc_Append}
%S_2+\Phi_0^\prime(R)\rho_2=0\quad \text{on}\ \partial B_R,
%\end{equation}
%\begin{equation}
%\label{SecondOrderLiouv_2bc_Append} 0=
%\partial_r S_2+\Phi_0^{\prime\prime}(R)\rho_2-\frac{\beta}{R^2}(\rho_2^{\prime\prime}+\rho_2)
%+\lambda_2 \quad \text{on}\ \partial B_R.
%\end{equation}
Similarly to above considerations, applying  Fourier analysis to problem
\eqref{SecondOrderLiouvAppend},\eqref{SecondOrderLiouv_1bc_Append},
\eqref{SecondOrderLiouv_2bc_Append} we find
\begin{equation}
\label{SecondOrderS2andphi_Append} S_2=\Lambda_1\partial_\Lambda
\tilde \phi(r,\Lambda_0) \cos\varphi +\tilde S_2(r)\cos
2\varphi+G(r), \quad \rho_2= -\frac{\tilde
	S_2(R)}{\Phi_0^\prime(R)}\cos 2\varphi,
\end{equation}
where $\tilde S_2$ solves
\begin{equation}
-\tilde S_2^{\prime\prime}-\frac{1}{r}\tilde
S_2^\prime+(1+4/r^2)\tilde S_2  - \Lambda_0 e^{\Phi_0(r)}\tilde
S_2 = \frac{\Lambda_0}{4} e^{\Phi_0(r)} (\tilde \phi(r,\Lambda_0)-
r)^2 \label{eq:secondorder_stress_bis}
\end{equation}
on $(0,R)$ with
\begin{equation}
\tilde S_2(0)=0,\quad \tilde
S_2^\prime(R)=\frac{\Phi_0^{\prime\prime}(R)-3\beta/R^2}{\Phi_0^\prime(R)}\tilde
S_2(R), \label{eq:secondorder_stress_bis_bc}
\end{equation}
and $G(r)$ is some function whose particular form is not important
for the further analysis. Note that under the condition
\eqref{exceptional_va_Liouv} problem
\eqref{eq:secondorder_stress_bis}-\eqref{eq:secondorder_stress_bis_bc}
has a unique solution.

Considering the Fourier mode corresponding to $\cos\varphi$  in \eqref{SecondOrderLiouv_2bc_Append} we obtain that
$\Lambda_1=0$, provided that $\partial_\Lambda  \tilde
\phi^\prime(R,\Lambda_0)\not=0$. The latter inequality is proved
as follows. Multiply \eqref{equatttion81}  by
$\Phi^\prime(r,\Lambda)r$ and integrate from $0$ to $R$ to find
that
\begin{equation}
\label{Proizv_vspomog_Append}
\tilde\phi^\prime(R,\Lambda)=\frac{\Lambda}{R\Phi^\prime(R,\Lambda)}\int_0^R
e^{\Phi(r,\Lambda)}\Phi^\prime (r,\Lambda)r^2\, dr=1+ \frac
{\Lambda R^2-\int_0^R \Phi(r,\Lambda)r\, dr-\Lambda \int_0^R
	e^{\Phi(r,\Lambda)}r\, dr} { \int_0^R \Phi(r,\Lambda)r\, dr
	-\Lambda \int_0^R e^{\Phi(r,\Lambda)}r\, dr }.
\end{equation}
Then
\begin{equation}
\label{Proizv_vspomog_AppendBIS}
\partial_\Lambda \tilde\phi^\prime(R,\Lambda_0)>
\frac
{R^2-\int_0^R \partial_\Lambda \Phi(r,\Lambda_0)r\, dr- \int_0^R
	e^{\Phi(r,\Lambda_0)}r\, dr} { \int_0^R \Phi(r,\Lambda_0)r\, dr
	-\Lambda_0 \int_0^R e^{\Phi(r,\Lambda_0)}r\, dr },
\end{equation}
where we have used the fact that  minimal solutions
$\Phi(r,\Lambda)$  are  increasing in $\Lambda$ and the
denominator in \eqref{Proizv_vspomog_Append} is negative. Since the pair $(\Lambda,\Phi)=(\Lambda_0,\Phi_0)$ satisfies \eqref{bifurccondition2} we have
\begin{equation}
\label{Proizv_vspomog_AppendBIS}
\partial_\Lambda\tilde \phi^\prime(R,\Lambda_0)>-
\frac
{\int_0^R (\partial_\Lambda \Phi(r,\Lambda_0)- \Phi(r,\Lambda_0)/\Lambda_0)r\, dr} { \int_0^R \Phi(r,\Lambda_0)r\, dr
	-\Lambda_0 \int_0^R e^{\Phi(r,\Lambda_0)}r\, dr }.
\end{equation}
Furthermore we obtain
that the function $w=\partial_\Lambda \Phi(r,\Lambda_0)- \Phi(r,\Lambda_0)/\Lambda_0$
is positive applying the maximum principle to the
equation $-\Delta w+w=\Lambda_0 e^{\Phi_0(r)}\partial_\Lambda\Phi(r,\Lambda_0)>0$.
Thus $\partial_\Lambda \tilde \phi^\prime(R,\Lambda_0)>0$.

Finally, to identify $S_3$ and $\rho_3$ we apply Fourier
analysis to
\eqref{eq:thirdorder_stress_bisAppend},\eqref{ThirdOrderLiouv_1bc_Append},
\eqref{ThirdOrderLiouv_2bc_Append}. The resulting formula for $\rho_3$ is
\begin{equation}
\label{rho3_Liouv_append} \rho_3=-\frac{\tilde
	S_3(R)}{\Phi_0^\prime(R)}\cos 3\varphi,
\end{equation}
where $\tilde S_3$ is the solution of the equation
\begin{equation}
-\tilde S_3^{\prime\prime}-\frac{1}{r}\tilde
S_3^\prime+(1+9/r^2)\tilde S_3  - \Lambda_0 e^{\Phi_0(r)}\tilde
S_3 = \frac{\Lambda_0}{2} e^{\Phi_0(r)} (\tilde \phi(r)- r) \tilde
S_2(r)+ \frac{\Lambda_0}{24} e^{\Phi_0(r)} (\tilde \phi(r)- r)^3
\label{eq:thirdorder_stress_bis}
\end{equation}
on $(0,R)$ with boundary conditions
\begin{equation}
\tilde S_3(0)=0,\quad \tilde S_3^\prime(R)=
\frac{\Phi_0^{\prime\prime}(R)-8\beta/R^2}{\Phi_0^\prime(R)}\tilde
S_3(R) +\frac{\tilde
	\phi^{\prime\prime}(R)-2/R}{2\Phi_0^\prime(R)}\tilde S_2(R).
\label{eq:thirdorder_stress_bis_bc}
\end{equation}
Thus the first terms of the asymptotic expansion of the function $\rho$ which determine the shape of the 
domain are
\begin{equation}
\rho=-\ve^2\frac{\tilde
	S_2(R)}{\Phi_0^\prime(R)}\cos 2\varphi-\ve^3\frac{\tilde
	S_3(R)}{\Phi_0^\prime(R)}\cos 3\varphi+\dots
\end{equation}
where $\tilde S_2$ solves 
\eqref{eq:secondorder_stress_bis}-\eqref{eq:secondorder_stress_bis_bc}, $\tilde S_3$ solves 
\eqref{eq:thirdorder_stress_bis}-\eqref{eq:thirdorder_stress_bis_bc} and $\tilde \phi$ is a solution of \eqref{equatttion81}-\eqref{equatttion81BC} with $\Lambda=\Lambda_0$ 
and $\Phi=\Phi_0$.

\end{appendices}

\end{document}